\documentclass[a4paper, 10pt]{article}

\usepackage{mystyle_article}
\bibliography{references.bib}

\title{Convergence of Allen--Cahn equations to De Giorgi's multiphase mean curvature flow}
\author{Pascal Steinke\footnote{steinke@iam.uni-bonn.de, Institute for Applied Mathematics, University of Bonn, Endenicher Allee 60, D-53115 Bonn / Germany}}

\begin{document}

	
	\maketitle
	\begin{abstract}
	
	This paper presents a conditional convergence result of solutions to the 
	Allen--Cahn equation with arbitrary potentials to a De Giorgi type $ \bv 
	$-solution to multiphase mean 
	curvature flow. 
	Moreover we show that De Giorgi type $ \bv 
	$-solutions are De Giorgi type varifold solutions, 
	and thus our solution is unique in a weak-strong sense.
	\end{abstract}

	\section{Introduction}
	
	\subsection{History and main results}
	
	Multiphase mean curvature flow is an important geometric evolution equation 
	which 
	has been studied for a long time, bearing not only mathematical importance, but 
	also for the applied sciences. Originally it was proposed to study the 
	evolution of grain boundaries in annealed recrystallized metal, as described by
	Mullins in \cite{mullins_two_dimensional_motion_of_idealized_grain_boundaries}, 
	who cites Beck in \cite{beck_metal_interfaces} as already having observed such 
	a behaviour in 1952. 
	
	Over the years numerous different solution concepts for multiphase mean 
	curvature flow have been proposed. Classically we have smooth solutions, where 
	we require the evolution of the interfaces to be smooth, for example described 
	by Huisken 
	in \cite{huisken_asymptotic_behavior_for_singuliarities_of_mcf}. 
	Another description of smoothly evolving mean curvature flow can be found in 
	the work of Gage and Hamilton 
	\cite{gage_hamilton_the_heat_equation_shrinking_convex_plane_curves}, who 
	proved the \enquote{shrinking conjecture} for convex planar curves.
	Brakke describes in his book 
	\cite{brakke_kenneth_motion_of_surface_by_mean_curvature} the motion by mean 
	curvature using varifolds, which yields a quite abstract and general notion for 
	mean curvature flow and is based on the gradient flow structure of mean 
	curvature flow. 
	Luckhaus and Sturzenhecker introduced a 
	distributional solution concept for mean curvature flow in their work 
	\cite{luckhaus_sturzenhecker_implicit_time_discretization_for_mcf}. Another 
	approach is the viscosity solution concept presented in 
	\cite{chen_giga_goto_uniqueness_and_existence_of_generalized_mcf_equations} and
	\cite{evans_spruck_motion_of_level_sets_by_mean_curvature}, where it is shown 
	that solutions of a certain parabolic equation have the property that if they 
	are smooth, the corresponding level sets move by mean curvature.
	
	For some smooth potential $ W 
	\colon \mathbb{ R }^{ N } \to [ 0 , \infty ) $ with 
	finitely many zeros $ \alpha_{ 1 } , \dotsc, \alpha_{ P } $, the Allen--Cahn 
	equation
	\begin{equation}
	\label{ac_intro}
	\partial_{ t } u_{ \varepsilon }
	=
	\Delta u_{ \varepsilon }
	-
	\frac{ 1 }{ \varepsilon^{ 2 } }
	\nabla W ( u_{ \varepsilon } )
	\end{equation}
	is commonly used as a phase-field approximation for mean curvature flow and was 
	first proposed by Allen and Cahn in their paper 
	\cite{allen_cahn_microscopig_theory_for_antiphase_boundary_motion}.
	Since then it has been an extensively researched topic, both from a numerical 
	and 
	analytical standpoint. 
	
	If we consider the two-phase case, which corresponds to setting $ N = 1 $ and $ 
	P = 2 $, then 
	the behaviour of the solutions to (\ref{ac_intro}) as $ \varepsilon $ tends to 
	zero are thoroughly researched.
	The authors Bronsard and Kohn showed in 
	\cite{bronsard_kohn_motion_by_mean_curvature_as_singular_limit} the compactness 
	of solutions and regularity for the limit by exploiting the gradient flow 
	structure of (\ref{ac_intro}), more precisely by mainly utilizing the 
	corresponding energy dissipation inequality. Moreover they studied the 
	behaviour 
	for radially symmetric initial data and showed that the limit moves by mean 
	curvature in this 
	case. 
	Chen has proven in 
	\cite{chen_generation_and_propagation_of_interfaces_for_reaction_diffusion_equations}
	that as long as the interface of the limit evolves smoothly, we have 
	convergence to (classical) mean curvature flow. 
	A similar result has been shown 
	by De Mottoni and Schatzmann in 
	\cite{de_mottoni_schatzmann_geometrical_evolution_of_developed_interfaces}, 
	where they showed short time convergence. Their strategy was to prove through a 
	spectral estimate that the solution $ u_{ 
		\varepsilon } $ already coincides with the formal asymptotic expansion up to an 
	error. Similar strategies, but for nonlinear Robin 
	boundary conditions with angle close to 90 degrees, can be found in the paper
	\cite{abels_moser_convergence_of_ac_with_nonlinear_robin_boundary_condition_to_mcf}
	by Abels and Moser.
	
	Ilmanen made fundamental contributions in 
	\cite{ilmanen_convergence_of_ac_to_brakkes_mcf} by proving the convergence to 
	Brakke's mean curvature flow as long as the initial conditions are 
	well-prepared by exploiting the equipartition of energies. However his methods 
	seem to only work in the two-phase case since he uses a comparison principle 
	whose generalization to the vectorial case is not clear.
	
	The multiphase case for both the convergence of the Allen--Cahn equation and 
	mean 
	curvature flow is much more involved and a topic of current research. For the 
	convergence of the Allen--Cahn equation, asymptotic expansions  with Neumann 
	boundary conditions 
	have been studied by Keller, Rubinstein and Sternberg in 
	\cite{keller_rubinstein_sternberg_fast_reaction_slow_diffusion}. Even though 
	they considered the vector-valued Allen--Cahn equation with a multiwell 
	potential, their analysis restricts to the parts of the interfaces where no 
	triple junctions appear. Bronsard and 
	Reitich however considered in 
	\cite{bronsard_reitich_on_three_phase_boundary_motion_and_singular_limit} a 
	formal asymptotic expansion which also takes triple junctions into account
	for the three-phase case. Moreover they proved short-time existence for the 
	three-phase boundary problem.
	
	The analysis of multiphase mean curvature flow has been studied for example by 
	Mantegazza, Novaga and Tortorelli in 
	\cite{mantegazza_novaga_tortorelli_motion_by_curvature_of_planar_networks}, 
	where they considered the planar case when only a single triple junction 
	appears, and their work has been extended to several triple junctions in 
	\cite{mantegazza_novaga_pluda_schule_evolution_of_networks_with_multiple_junctions}.
	Ilmanen, Neves and Schulze furthermore proved in 
	\cite{ilmanen_neves_schulze_on_short_time_existence_for_the_planar_network_flow}
	that even for non-regular initial data in the sense that Herring's angle 
	condition is not satisfied, we have short time existence by approximation 
	through regular networks. For long time existence, it has been shown by Kim and 
	Tonegawa in \cite{kim_tonegawa_on_the_mean_curvature_flow_of_grain_boundaries} 
	through a modification of Brakke's approximation scheme that one can obtain 
	non-trivial mean curvature flow even with singular initial data. Recently 
	Stuvard and Tonegawa improved this result in 
	\cite{stuvard_tonegawa_on_the_existence_of_canonical_multi_phase_brakke_flow},
	where they extended their findings to more general initial data and also showed 
	that their constructed Brakke flow is a $ \bv $-solution to mean curvature flow.

	The first main goal of this paper \Cref{convergence_to_de_giorgis_multiphase_mcf} is to prove a conditional convergence result 
	of solutions to the vectorial Allen--Cahn equation (\ref{ac_intro}) to a De 
	Giorgi type $ \bv $-solution of multiphase mean curvature flow in the sense of 
	\Cref{de_giorgi_solution_to_mmcf}.
	The proof is based mainly on a duality 
	argument and the results of Laux and Simon in 
	\cite{convergence_of_allen_cahn_equation_to_multiphase_mean_curvature_flow}.
	However the strong assumption we make here is that the Cahn--Hilliard energies 
	of the solutions to the Allen-Cahn equation (\ref{ac_intro}) converge to 
	the perimeter functional (\ref{definition_of_multiphase_energy}) applied to the 
	limit. 
	This prevents that as $ \varepsilon $ tends to zero, the approximate interfaces 
	collapse, which corresponds to a loss of energy in the limit $ \varepsilon \to 0 
	$, see also the discussion in
	\Cref{subsection_de_giorgi_type_varifold_solutions_for_mcf}.  In general we are 
	inclined to believe that this assumption could fail. For example it has been 
	shown by Bronsard and Stoth in 
	\cite{bronsard_stoth_on_the_existence_of_high_multiplicity_interfaces}
	that for the volume preserving Allen--Cahn equation and for radial-symmetric 
	initial data, we can have any number of higher 
	multiplicity transition layers which are at most $ C \varepsilon^{ \alpha } $ 
	apart, at least for times of order one.
	Here $ \alpha $ is some exponent between zero and one third.
	However it has been proven that in some cases, the energy convergence automatically holds for suitable initial data. For example Nguyen and Wang proved in 
	\cite{nguyen_wang_brakke_regularity_for_the_allen_cahn_flow}
	that for mean convex initial data, we obtain the desired convergence of energies. 
	In the context of the Almgren-Taylor-Wang scheme for mean curvature flow, this has also been done by De Philippis and Laux in \cite{de_philippis_laux_implicit_time_discretization_for_the_mean_curvature_flow_of_mean_convex_sets}. 
	
	The energy 
	convergence assumption provides us with the important equipartition of 
	energies, whose proof under milder assumption was the main obstacle of Ilmanen 
	in \cite{ilmanen_convergence_of_ac_to_brakkes_mcf}. Moreover it is the key for 
	our localization estimates and lets us localize on the different phases of the 
	limit. Lastly it assures that the differential $ \nabla u_{ 
		\varepsilon } $ can locally up to an error be written as a rank-one matrix. In 
	fact it
	is the tensor of the approximate frozen unit normal and the gradient of the 
	geodesic 
	distance function associated to the majority phase evaluated at $ u_{ 
		\varepsilon } $, see the proof of \cite[Prop.~3.1]{convergence_of_allen_cahn_equation_to_multiphase_mean_curvature_flow}.
	
	The second main goal is to compare De Giorgi type $ \bv $-solutions to De 
	Giorgi type varifold solutions.
	The latter were 
	proposed by Hensel and Laux in 
	\cite{hensel_laux_varifold_solution_concept_for_mean_curvature_flow}. We will 
	show that the De Giorgi type $ \bv $-solution concept is stronger in the sense 
	that every $ \bv 
	$-solution is also a varifold solution. 
	Since Hensel and Laux have shown 
	weak-strong uniqueness for their varifold solution concept, it follows that we 
	also obtain weak-strong uniqueness for our De Giorgi type $ \bv $-solution 
	concept, which is the best we can expect. 
	In fact it has been shown on a 
	numerical basis by Angenent, Chopp and Ilmanen in 
	\cite{angenent_chopp_ilmanen_a_computed_example_of_nonuniqueness_of_mcf}
	that even in three dimensions, there exists a smooth hypersurface whose 
	evolution by 
	mean curvature 
	flow admits a singularity at a certain time after which we have nonuniqueness. 
	A rigorous analysis of this phenomenon has been done in dimensions 4, 5, 6, 7 
	and 8 by Angenent, Ilmanen and Velázquez in 
	\cite{angenent_ilmanen_velázquez_fattening_from_smooth_initial_data_in_mcf}.
	
	Let us also mention some of the closely related unanswered questions. For one 
	Hensel and Laux have shown in 
	\cite{hensel_laux_varifold_solution_concept_for_mean_curvature_flow}
	that in the two-phase case and under well prepared initial conditions, the
	solutions of the Allen--Cahn equation (\ref{ac_intro}) converge to a De Giorgi 
	type varifold solution. However their methods have no obvious generalization to 
	the multiphase case without the energy convergence assumption, and one even 
	struggles to find an approximate sequence 
	which constructs the desired varifolds. And even then one would have to find 
	suitable substitutions for localization which will be crucial for our proof. 
	These localization estimates are based on De Giorgi's structure 
	theorem and thus only work for $ \bv $-functions.
	
	Another possible question would be how to generalize the results to the case 
	of arbitrary mobilities: Throughout the paper, and also the main background 
	papers 
	\cite{convergence_of_allen_cahn_equation_to_multiphase_mean_curvature_flow},
	\cite{hensel_laux_varifold_solution_concept_for_mean_curvature_flow}, it is 
	always assumed that the mobilities are fixed through the relation $ \mu_{ i j } 
	= 1/ \sigma_{ i j } $.
	Here $ \sigma_{ i j } $ denotes the surface tension of 
	the $ ( i , j ) $-th interface and $ \mu_{ i j } $ its mobility. As proposed by
	Bretin, Danescu, Penuelas and Masnou in 
	\cite{bretin_dansecu_penuelas_masnou_a_metric_based_approach_to_mmcf_with_mobilities},
	passing
	to arbitrary mobilities should 
	amount to multiplying an appropriate \enquote{mobility matrix} $ M $ onto the 
	right-hand side of (\ref{ac_intro}) and changing the metric of the underlying 
	space accordingly to $ \langle u , v \rangle = \int \inner*{M u}{ v} \dd{ x } $.
	The difference in their approach is to first uncouple their system so that 
	they arrive at the scalar Allen--Cahn equation and then couple the components 
	through a Lagrange-multiplier, which assures that the limit is a partition.
	
	\subsection{Structure of the paper}
	
	In \Cref{section_gf_and_mcf} we first describe the concept of gradient flows in an easy setting. Moreover we derive De Giorgi's 
	optimal energy dissipation inequality in a simple example and discuss its 
	usefulness for 
	reformulating the gradient flow equation. Further we apply our 
	observations 
	to (multiphase) mean curvature flow.
	
	Continuing with 
	\Cref{section_convergence_of_ac_equation_to_an_evolving_parititon_article}, we take a look at the Allen--Cahn equation and the convergence of its solutions as $ 
	\varepsilon $ tends to zero to an evolving partition by citing the results of \cite{convergence_of_allen_cahn_equation_to_multiphase_mean_curvature_flow}.
	
	Lastly in \Cref{chapter_de_giorgis_mcf} we present new results 
	building on the previous insights. We start off by presenting a De Giorgi type 
	$ \bv $-solution concept for multiphase mean curvature flow. This is followed 
	by proving a similar conditional convergence result as in the previous section. 
	In the final 
	\Cref{subsection_de_giorgi_type_varifold_solutions_for_mcf}, we discuss the 
	assumption of energy convergence. Moreover we show that every De 
	Giorgi type $ \bv $-solution is a De Giorgi type varifold solution in the sense 
	of Hensel and Laux in 
	\cite{hensel_laux_varifold_solution_concept_for_mean_curvature_flow}, whose 
	solution concept does not rely on the assumption of energy convergence.
	
	\subsection{Notation}

	\begin{itemize}[leftmargin=*]
		\item	
		Let $ u \colon \mathbb{ R }^{ d } \to \mathbb{ R }^{ N } $ be 
		differentiable at a point $ x \in \mathbb{ R }^{d } $. We write $ \diff 
		u ( x ) $ for the (total) derivative at $ x $, which means that
		$ \diff u ( x ) \in \mathbb{ R }^{ N \times d } $ and
		$ \left( \diff u ( x ) \right)_{ i j } = \partial_{ x_{ j } } u^{ i } ( 
		x ) $. We always use the notation $ \nabla u ( x ) $ for the transpose 
		of the total derivative. 
		
		\item
		If $ u \colon ( 0 , T ) \times \mathbb{ R }^{ d } \to \mathbb{ R 
		}^{ N } $, then we denote by $ \diff u (t, x ) $  respectively $ \nabla 
		u ( 
		t,  x ) $ only the derivatives in space, and we always write $ 
		\partial_{ 
			t } u $ for the derivative in time.
		
		\item
		We use the symbol $ \lesssim $ if an inequality holds up to a positive 
		constant on the right hand side. This constant must only depend on the 
		dimensions, the chosen potential $ W $ and possibly the given initial data.
		
		\item 
		The bracket $ \inner*{\:\cdot\:}{\:\cdot\:} $ is used as the Euclidean 
		inner 
		product. Depending on the situation, it acts on vectors or matrices.
		If we apply $ \abs{ \:\cdot\: } $ to a vector or matrix, then we always 
		uses the norm induced by the Euclidean inner product.
		
		\item 
		We denote by $ \cont_{ \mathrm{c} } $ the compactly supported 
		continuous functions. Note that for the flat torus $ \flattorus $, we 
		have $ \cont_{ \mathrm{c} } ( \flattorus ) = \cont ( \flattorus ) $.
		
		\item
		For a locally integrable function $ u \colon ( 0 , T ) \times \Omega 
		\to \mathbb{ R } $ 
		with $ \Omega \subseteq \mathbb{ R }^{ d } $ being some open set or the 
		flat torus, we denote the total variation in space for a time $ t \in ( 
		0 , T ) $ by
		\begin{equation*}
		\abs{ \nabla u ( t , \cdot ) }
		\coloneqq
		\sup 
		\left\{
		\int_{ \Omega }
		u ( t , x ) \divg \xi ( x ) 
		\dd{ x }
		\, \colon \,
		\xi \in \cont_{ \mathrm{c} }^{ 1 } ( \Omega ; \mathbb{ R }^{ d 
		} ) , \abs{ \xi } \leq 1
		\right\}
		\end{equation*}	
		and the total variation in space taken both over time and space by
		\begin{align*}
		& \abs{ \nabla u }_{ d + 1 }
		\\
		\coloneqq{} &
		\left\{
		\int_{ ( 0 , T ) \times \Omega }
		u ( t , x ) \divg_{ x } \xi ( t , x ) 
		\dd{ ( t, x ) }
		\, \colon \,
		\xi \in \cont_{ \mathrm{c} }^{ 1 } (( 0 , T ) \times \Omega ; 
		\mathbb{ R }^{ d 
		} ) , \abs{ \xi } \leq 1
		\right\}.
		\end{align*}
		
		\item 
		If $ u \colon ( 0 , T ) \times \Omega \to \mathbb{ R }^{ N } $ is some 
		map, then for a given time $ t \in ( 0 , T ) $, we sometimes write $ 
		u ( t ) $ for the map $ u ( t ) ( x ) \coloneqq u ( t , x ) $ by a 
		slight abuse of notation.
		
		\item
		To shorten notation, we always write $ \int \dd{ x } = \int_{ 
			\flattorus } \dd{ x } $, where $ \flattorus $ is the flat torus.
	\end{itemize}


	
	\section{Gradient flows and mean curvature flow}
	\label{section_gf_and_mcf}
	\subsection{Gradient flows}
	\label{section_gradient_flows}
	In the simplest case, a gradient flow of a given energy $ \energy \colon 
	\mathbb{ R }^{ N } \to \mathbb{ R } $ with respect to the Euclidean inner 
	product is a solution to the ordinary differential equation
	\begin{equation}
	\label{gf_basic_equation}
	\dv{ t } x ( t ) = - \nabla \energy ( x ( t ) ) ,
	\end{equation}
	where we usually prescribe some initial value $ x ( 0 ) = x_{ 0 } \in \mathbb{ 
		R }^{ N } $. The central structure here is that on the right hand side of the 
	equation, we do not have an arbitrary vector field, but the gradient of some 
	continuously 
	differentiable function. Remember that the gradient of a function always 
	depends on the chosen metric of our space. In fact, the gradient $ \nabla 
	\energy ( x ) $ is always the unique element of the tangent space at $ x $ 
	which 
	satisfies
	\begin{equation*}
	\inner*{ \nabla \energy ( x ) }{ y }
	=
	\dd_{ x } \energy ( y ) 
	\end{equation*}
	for all elements $ y $ in the tangent space at $ x $. In this simple setting, the 
	tangent space at any point is exactly $ \mathbb{ R }^{ N } $.  
	
	A solution $ x $ moves in the 
	direction of the steepest descent of the energy $ \energy $. Moreover this 
	allows for the following computation given a continuously differentiable 
	solution $ x $ of (\ref{gf_basic_equation}):
	\begin{align*}
	\dv{ t } \energy ( x ( t ) ) 
	& =
	\inner*{ \nabla \energy ( x ( t ) ) }{ x ' ( t ) }
	=
	- \abs{ \nabla \energy ( x ( t ) ) }^{ 2 }
	=
	- \abs{ x' ( t ) }^{ 2 }
	\\
	& =
	- \frac{ 1 }{ 2 }
	\left(
	\abs{ x'( t ) }^{ 2 }
	+
	\abs{ \nabla \energy ( x ( t ) ) }^{ 2 }
	\right).
	\end{align*}
	We especially obtain that the function $ \energy ( x ( t ) ) $ is 
	non-increasing, which coincides with our intuition of the steepest descent. But 
	more precisely, we obtain from the 
	fundamental theorem of calculus the \emph{energy dissipation identity}
	\begin{equation}
	\label{basic_energy_dissipation_identity}
	\energy ( x ( T ) )
	+
	\frac{ 1 }{ 2 }
	\int_{ 0 }^{ T }
	\abs{ x' ( t ) }^{ 2 }
	+
	\abs{ \nabla \energy ( x ( t ) ) }^{ 2 }
	\dd{ t }
	= 
	\energy ( x ( 0 ) ).
	\end{equation}
	One could now raise the question if this identity already characterizes 
	equation (\ref{gf_basic_equation}). 
	But as it turns out, we can go even one step further. Namely we only ask for 
	\emph{De Giorgi's optimal energy dissipation inequality} given by
	\begin{equation}
	\label{basic_optimal_energy_dissipation_inequality}
	\energy ( x ( T ) )
	+
	\frac{ 1 }{ 2 }
	\int_{ 0 }^{ T }
	\abs{ x' ( t ) }^{ 2 }
	+
	\abs{ \nabla \energy ( x ( t ) ) }^{ 2 }
	\dd{ t }
	\leq
	\energy ( x ( 0 ) ).
	\end{equation}
	We call this inequality optimal since as demonstrated before, we usually expect 
	an equality to hold if $ x $ is a solution to (\ref{gf_basic_equation}).
	Now let us assume that $ x $ satisfies inequality
	(\ref{basic_optimal_energy_dissipation_inequality}) and is sufficiently regular.
	Then we can estimate again by the fundamental theorem of calculus that
	\begin{align*}
	& \frac{ 1 }{ 2 }
	\int_{ 0 }^{ T }
	\abs{
		\nabla \energy ( x ( t ) )
		+
		x' ( t ) 
	}^{ 2 }
	\dd{ t }
	\\
	={} &
	\int_{ 0 }^{ T }
	\inner*{ \nabla \energy ( x ( t ) ) }{ x' ( t ) }
	\dd{ t }
	+
	\frac{ 1 }{ 2 }
	\int_{ 0 }^{ T }
	\abs{ x' ( t ) }^{ 2 }
	+
	\abs{ \nabla \energy ( x ( t ) ) }^{ 2 }
	\dd{ t }
	\\
	={} &
	\energy ( x ( T ) ) - \energy ( x ( 0 ) ) 
	+
	\frac{ 1 }{ 2 }
	\int_{ 0 }^{ T }
	\abs{ x' ( t ) }^{ 2 }
	+
	\abs{ \nabla \energy ( x ( t ) ) }^{ 2 }
	\dd{ t }
	\\
	\leq{} & 0.
	\end{align*}
	Since we started with an integral over a non-negative function, this implies 
	that for almost every time $ t $, we have that $ x' ( t ) = - \nabla \energy ( 
	x ( t ) ) $. But if $ x $ and $ E $ are sufficiently regular, this already 
	implies that $ x' ( t ) = - \nabla \energy ( x ( t ) ) $ holds for all times $ 
	t 
	$.
	Thus $ x $ is a gradient flow of 
	the energy $ \energy $.
	
	The real strength of formulating the differential equation 
	(\ref{gf_basic_equation}) via inequality 
	(\ref{basic_optimal_energy_dissipation_inequality}) becomes clear if we want to 
	consider gradient flows in a more complicated setting. In order to formulate 
	equation (\ref{gf_basic_equation}), we need to have a notion of differentiation 
	and a gradient in the target of $ x $. Therefore we may use pre-Hilbert spaces 
	or smooth Riemannian manifolds as suitable substitutes for $ \mathbb{ R }^{ N 
	} $. 
	
	Examples for such more complicated gradient flows include the heat equation, 
	which can be written as the $ \lp^{ 2 } $-gradient flow of the Dirichlet 
	energy. 
	Further examples include  
	the Fokker-Planck equation, the Allen--Cahn equation and mean curvature flow. 
	The two latter of course play a key role for us.
	
	The observations in this section have been around for a long time and are 
	credited 
	to De Giorgi and his paper 
	\cite{de_giorgi_new_problems_on_minimizing_movements}. Sandier and Serfaty have 
	also written an excellent paper 
	\cite{sandler_serfaty_gamme_convergence_of_gf_with_applications_to_gl} on this 
	topic, as well as the book 
	\cite{ambrosio_gigli_savare_gradient_flows_in_metric_spaces_and_in_the_space_of_prob_measures}
	by Ambrosio, Gigli and 
	Savaré,
	to which we refer the 
	interested reader.
	
	\subsection{Mean curvature flow}
	\label{section_mcf}
	
	Mean curvature flow describes the geometric evolution of a set $ ( \Omega ( t ) 
	)_{ t \geq 0 } $ respectively the evolution of its boundary $ \Sigma ( t ) 
	\coloneqq \partial \Omega ( t ) $. It is formulated through the equation
	\begin{equation}
	\label{mcf_twophase_basic_equation}
	\frac{ 1 }{ \mu } V = - \sigma H 
	\quad
	\text{on }
	\Sigma.
	\end{equation}
	Here $ \mu > 0 $ is a positive constant which is called the \emph{mobility} and 
	$ \sigma > 0 $ is a positive constant as well which we define as the 
	\emph{surface tension}. By $ V $ we denote the normal velocity of the set and 
	by $ H $ its mean curvature, which is defined as the sum of the principle 
	curvatures at a given point.
	There are a lot of different notions of solutions to this equation as already 
	mentioned in the introduction.
	For us the central structure of this equation is its gradient flow structure. Formally we want to consider the space
	\begin{equation*}
	\mathcal{ M }
	\coloneqq
	\left\{
	\text{hypersurfaces in} \flattorus
	\right\},
	\end{equation*}
	where the tangent space at a given $ \Sigma \in \mathcal{ M } $ consists of the normal velocities on $ \Sigma $.
	The metric tensor at a surface $ \Sigma $ of two normal velocities is then 
	given by the rescaled $ \lp^{ 2 } $-inner product on $ \Sigma $
	\begin{equation}
	\label{definition_of_metric_twophase}
	\inner*{ V }{ W }_{ \Sigma }
	\coloneqq
	\frac{ 1 }{ \mu }
	\int_{ \Sigma }
	V W
	\dd{ \hm^{ d - 1 } }
	\end{equation}
	and our energy will simply be the rescaled perimeter functional
	\begin{equation*}
	\energy ( \Sigma )
	\coloneqq
	\sigma \hm^{ d -1 } ( \Sigma ).
	\end{equation*}
	By \cite[Thm.~17.5]{maggi_sets_of_finite_perimeter} the first inner variation 
	of the perimeter functional is given by the mean curvature vector. 
	Remember moreover that by the definition of the metric 
	(\ref{definition_of_metric_twophase}) the gradient of the energy at a given 
	hypersurface $ \Sigma $ has to satisfy
	\begin{equation*}
	\frac{ 1 }{ \mu }
	\int_{ \Sigma }
	\inner*{ \nabla_{ \Sigma } \energy }{ V }
	\dd{ \hm^{ d - 1 } }
	=
	\dd_{ \Sigma }
	\energy ( V )
	\end{equation*} 
	for all normal vector fields $ V $ on $ \Sigma $.
	Combining these arguments the gradient of the energy should simply be the mean 
	curvature vector multiplied by $ \sigma \mu $.
	Thus the mean curvature flow equation (\ref{mcf_twophase_basic_equation}) 
	corresponds exactly to the gradient flow equation (\ref{gf_basic_equation}).
	Note however that this metric tensor induces a degenerate metric in the sense 
	that the distance between any two hypersurfaces is zero, which has been shown 
	by Michor and Mumford 
	in \cite{michor_mumford_riemannian_geometries_on_spaces_of_plane_curves}. 
	
	In \Cref{section_gradient_flows} we highlighted the importance of
	De Giorgi's optimal energy dissipation inequality (\ref{basic_optimal_energy_dissipation_inequality}). In our setting  this translates to the inequality
	\begin{align*}
	& \energy ( \Sigma ( T ) )
	+
	\frac{ 1 }{ 2 }
	\int_{ 0 }^{ T }
	\inner*{ V ( t ) }{ V ( t ) }_{ \Sigma ( t ) }
	+
	\inner*{ \nabla_{ \Sigma ( t ) } E }{ \nabla_{ \Sigma ( t ) } E }_{ 
		\Sigma ( t ) }
	\dd{ t }
	\\
	={}
	&\energy ( \Sigma ( T ) )
	+
	\frac{ 1 }{ 2 }
	\int_{ 0 }^{ T }
	\int_{ \Sigma (t ) }
	\frac{ 1 }{ \mu }
	V( t )^{ 2 } 
	+
	\sigma^{ 2 } \mu 
	H( t ) ^{ 2 }
	\dd{ \hm^{ d - 1 } }
	\dd{ t }
	\\
	\leq{} &
	\energy ( \Sigma ( 0 ) ).
	\end{align*}
	This is the main motivation for the definition of De Giorgi type solutions to 
	mean curvature flow in the two-phase case, see 
	\Cref{de_giorgi_solution_to_mmcf} and 
	\Cref{de_giorgi_varifold_solution_for_mcf}.
	
	Far more important for us shall however be multiphase mean curvature flow, 
	which is of high importance in the applied sciences and mathematically quite 
	complex. 
	Essentially instead of just considering the evolution of one single set, we 
	look at the evolution of a partition of the flat torus and require that at each 
	interface of the sets, the mean curvature flow equation 
	(\ref{mcf_twophase_basic_equation}) is satisfied. Moreover we require a 
	stability condition at points where three interfaces meet.
	
	More precisely we say that a partition $ ( \Omega_{ i } ( t ) )_{ i = 1 , 
		\dotsc , P } $ with 
	$ \Sigma_{ i j } \coloneqq \partial \Omega_{ i } \cap \partial \Omega_{ j } $ 
	satisfies multiphase mean curvature flow with mobilities $ \mu_{ i j } $ and 
	surface 
	tensions $ \sigma_{ i j } $ if 
	\begin{align}
	\label{v_is_equal_to_h}
	\frac{ 1 }{ \mu_{ i j } }
	V_{ i j }
	& =
	-
	\sigma_{ i  j }
	H_{ i j }
	\quad
	&\text{on } \Sigma_{i j }\text{ for all }i\neq j \text{ and}
	\\
	\label{herrings_angle_condition}
	\sigma_{ i j }
	\nu_{ i j }
	+
	\sigma_{ j k }
	\nu_{ j k }
	+
	\sigma_{ k i }
	\nu_{ k i }
	& =
	0
	& \text{at triple junctions}.
	\end{align}
	The second equation is a stability condition when more than two sets meet and is
	called \emph{Herring's angle condition}. Here $ \nu_{ i j } $ is the outer unit 
	normal of $ \Omega_{ i } $
	on $ \Sigma_{ i j } $ pointing towards $ \Omega_{ j } $. 
	One could argue that we would also want stability conditions at for example 
	quadruple 
	junctions. In two dimensions though, such quadruple junctions are expected to 
	immediately dissipate. In higher dimensions, quadruple junctions might become 
	stable, but only on lower dimensional sets, and shall thus not be relevant for 
	us.
	
	As our space, we shall therefore now consider tuples of hypersurfaces in $ 
	\flattorus $. 
	The energy and metric tensor are given by
	\begin{align}
	\label{definition_of_multiphase_energy}
	\energy ( \Sigma )
	&\coloneqq
	\sum_{ i < j }
	\sigma_{ i j }
	\hm^{ d- 1 } ( \Sigma_{ i j } )
	\shortintertext{and}
	\notag
	\inner*{ V }{ W }_{ \Sigma }
	& \coloneqq
	\sum_{ i < j }
	\frac{ 1 }{ \mu_{ i j } }
	\int_{ \Sigma_{ i j } }
	V_{ i j }
	W_{ i j }
	\dd{ \hm^{ d - 1 } }.
	\end{align}
	Of course this again produces a degenerate metric.
	Using a variant of the divergence theorem on surfaces 
	(\cite[Thm.~11.8]{maggi_sets_of_finite_perimeter}) and again the computation 
	for the first variation of the perimeter 
	(\cite[Thm.~17.5]{maggi_sets_of_finite_perimeter}), we see that multiphase mean 
	curvature flow has the desired gradient flow structure.
	In this case De Giorgi's 
	inequality (\ref{basic_optimal_energy_dissipation_inequality}) translates to
	\begin{equation*}
	\energy ( \Sigma ( T ) )
	+
	\frac{ 1 }{ 2 }
	\sum_{ i < j }
	\int_{ 0 }^{ T }
	\int_{ \Sigma_{ i j } }
	\frac{ 1 }{ \mu_{ i j } }
	V_{ i j }^{ 2 }
	+
	\sigma_{ i j }^{ 2 } \mu_{ i j }
	H_{ i j }^{ 2 }
	\dd{ \hm^{ d - 1 } }
	\dd{ t }
	\leq
	\energy ( \Sigma ( 0 ) ).
	\end{equation*}
	Furthermore if we make the simplifying assumption that the mobilities are 
	already determined through the surface tensions by the relation $ \mu_{ i j } = 
	1 / \sigma_{ i j } $, then this inequality becomes
	\begin{equation*}
	\energy ( \Sigma ( T ) )
	+
	\frac{ 1 }{ 2 }
	\sum_{ i < j }
	\sigma_{ i j }
	\int_{ 0 }^{ T }
	\int_{ \Sigma_{ i j } }
	V_{ i j }^{ 2 }
	+
	H_{ i j }^{ 2 }
	\dd{ \hm^{ d - 1 } }
	\dd{ t }
	\leq
	\energy ( \Sigma ( 0 ) ),
	\end{equation*}
	which motivates \Cref{de_giorgi_solution_to_mmcf} and 
	\Cref{de_giorgi_varifold_solutions_for_mmcf}.
	
	\section{Convergence of the Allen-Cahn equation to an evolving partition}
	\label{section_convergence_of_ac_equation_to_an_evolving_parititon_article}
	\subsection{Allen-Cahn equation}
	
	This chapter follows 
	\cite{convergence_of_allen_cahn_equation_to_multiphase_mean_curvature_flow}.
	Let $ \Lambda > 0 $ and define the flat torus as the quotient
	$ \flattorus \coloneqq [0, \Lambda )^{ d } = \mathbb{ R }^{ d } / \Lambda \mathbb{ Z 
	}^{ d } $, 
	which means that we impose periodic boundary conditions on $ [ 0 , \Lambda )^{ d } $.
	Then for 
	$ u \colon [ 0 , \infty ) \times \flattorus \to \mathbb{ R }^{ N } $ 
	and some smooth potential 
	$ W \colon \mathbb{ R }^{ N } \to [0, \infty ) $,
	the \emph{Allen--Cahn equation} with parameter $ \varepsilon > 0 $ is given by
	\begin{equation}
	\label{allen_cahn_eq}
	\partial_{ t } u 
	=
	\Delta u - \frac{1 }{ \varepsilon^{ 2 } } \nabla W ( u ).
	\end{equation}
	To understand this equation better, we consider the \emph{Cahn--Hilliard 
		energy}, which assigns to $ u $ for a fixed time the value
	\begin{equation}
	\label{cahn_hilliard_energy}
	\energy_{ \varepsilon } 
	(u)
	\coloneqq
	\int
	\frac{ 1 }{ \varepsilon }
	W ( u )
	+
	\frac{ \varepsilon }{ 2 }
	\abs{ \nabla u }^{ 2 }
	\dd{x}.
	\end{equation}
	As it turns out the partial differential equation 
	(\ref{allen_cahn_eq}) is the $ \lp^{ 2 } $-gradient flow rescaled by $ 
	\sqrt{\varepsilon} $ of the Cahn--Hilliard 
	energy. Thus a solution to 
	(\ref{allen_cahn_eq}) can be constructed via De Giorgi's minimizing movements scheme 
	(\cite{de_giorgi_new_problems_on_minimizing_movements}).
	
	But first we need to clarify what our potential $ W $ should look like. Classical
	examples in the scalar case are given by $ W ( u ) = \left( u^{ 2 } - 1 
	\right)^{ 2 } $ or $ W( u ) = u^{ 2 } ( u - 1 )^{ 2 } $, and we call functions 
	like these \emph{doublewell potentials}, see also 
	\Cref{graph_of_doublewell_potential}.
	
	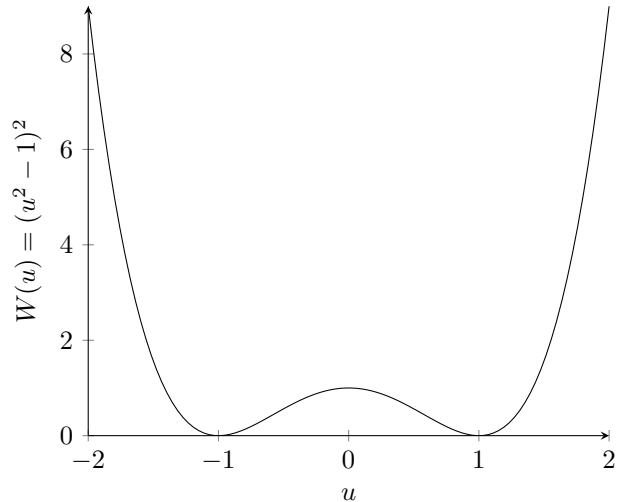
\begin{figure}
		\centering
		\begin{tikzpicture}
		\begin{axis}[
		axis lines = left,
		xlabel = \(u\),
		ylabel = {\(W(u)= (u^2 - 1)^{ 2 }\)},
		]
		\addplot [
		domain=-2:2, 
		samples=100, 
		color=black,
		]
		{x^4-2*x^2 + 1};
		\end{axis}
		\end{tikzpicture}
		\caption{Plot of the doublewell potential $ W(u) = (u^2 - 1 )^{ 2 } $}
		\label{graph_of_doublewell_potential}
	\end{figure}
	
	In higher dimensions, we want to accept potentials $ W \colon 
	\mathbb{ R }^{ N } \to [0, \infty ) $ which are smooth multiwell potentials, 
	meaning that they have finitely many zeros at $ u = \alpha_{ 1 }, \dotsc , 
	\alpha_{ P } \in \mathbb{ R }^{ N } $. Furthermore we ask for polynomial growth 
	in the sense that there exists some $ p \geq 2 $ such that for all $ u $ 
	sufficiently large, we have
	\begin{equation}
	\label{polynomial_growth}
	\abs{ u }^{ p } \lesssim W(u) \lesssim \abs{ u }^{ p }
	\end{equation}
	and
	\begin{equation}
	\label{polynomial_growth_derivative}
	\abs{ \nabla W ( u ) } \lesssim \abs{ u }^{ p -1 }.
	\end{equation}
	Lastly we want $ W $ to be convex up to a small perturbation in the sense that 
	there exist smooth functions 
	$ W_{ \mathrm{conv} }$, $ W_{ \mathrm{pert} } \colon \mathbb{ R }^{ N } \to [ 0 , \infty ) $ such that
	\begin{equation}
	\label{decomposition_of_w}
	W = W_{ \mathrm{conv}} + W_{ \mathrm{pert}}\, ,
	\end{equation}
	$ W_{ \mathrm{conv} } $ is convex and
	\begin{equation}
	\label{perturbation bound}
	\sup_{ x \in \mathbb{ R }^{ N } }
	\abs{ \nabla^{ 2 } W_{ \mathrm{pert} } } < \infty.
	\end{equation}
	One can check that the above mentioned doublewell potentials satisfy these 
	assumptions. In general one may consider potentials $ W $ of the form
	\begin{equation*}
	W ( u ) \coloneqq
	\abs{ u - \alpha_{ 1 } }^{ 2 } 
	\cdot
	\dotsc
	\cdot
	\abs{ u - \alpha_{ P } }^{ 2 },
	\end{equation*}
	see also \Cref{twodimensional_potential} for the case $ P = 3 $ in two 
	dimensions.
	\begin{figure}
		\centering
		
		\begin{subfigure}[b]{0.45\linewidth}
			
			\includegraphics[width=\textwidth]{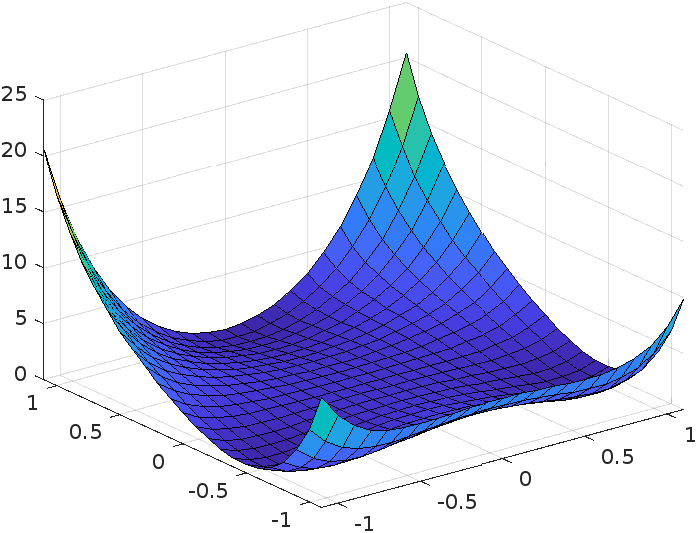}
			\caption{Graph of the potential}
		\end{subfigure}
		\hfill
		\begin{subfigure}[b]{0.45\linewidth}
			\includegraphics[width=\textwidth]{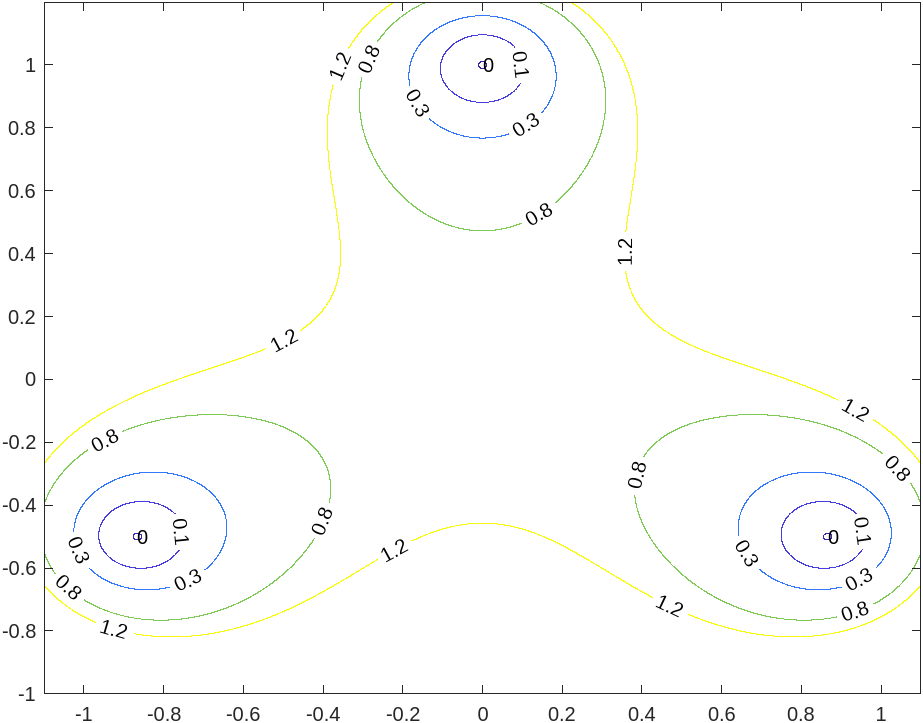}
			
			\caption{Contour lines of the potential}
		\end{subfigure}
		
		\caption{Graphics for the potential
			$ W ( u ) = 
			\abs{ u  - e^{ i \pi\frac{ 3 }{ 6 }} }^{ 2 } 
			\abs{ u  - e^{ i \pi\frac{   7}{ 6  }} }^{ 2 }
			\abs{ u  - e^{ i \pi\frac{  11}{ 6  }} }^{ 2 }$ }
		\label{twodimensional_potential}
	\end{figure} 
	The existence of the function $ W_{ \mathrm{pert} } $ can then be seen by 
	multiplying some suitable cutoff to $ W 
	$ which is equal to 1 on a sufficiently large ball.
	
	As it is custom for parabolic partial differential equations, we view solutions 
	of the Allen--Cahn equation (\ref{allen_cahn_eq}) as maps from $ [0,T] $ into 
	some suitable function space and thus use the following definition.
	
	\begin{definition}
		\label{solution_to_ac}
		We say that a function 
		\begin{equation*}
		u_{ \varepsilon} \in 
		\mathrm{ C } \left( [0 , T ] ; \lp^{ 2 } \left( \flattorus; \mathbb{ R }^{ 
			N } \right) \right) \cap
		\lp^{ \infty } \left( [0, T ]; \wkp^{ 1, 2 } ( \flattorus ; \mathbb{ R }^{ 
			N } ) \right)
		\end{equation*}
		is a weak solution of the Allen--Cahn equation (\ref{allen_cahn_eq}) with parameter $ \varepsilon > 0 $ and initial condition $ u_{ \varepsilon}^{ 0 } \in \lp^{ 2 } ( \flattorus ; \mathbb{ R }^{ N } ) $ if
		\begin{enumerate}
			\item the energy stays bounded, which means that
			\begin{equation}
			\esssup_{ 0 \leq t \leq T }
			\energy_{ \varepsilon } ( u_{ \varepsilon } ( t ) ) 
			< \infty,
			\end{equation}
			\item 
			its weak time derivative satisfies
			\begin{equation}
			\partial_{ t } u_{ \varepsilon }
			\in
			\lp^{ 2 } \left( [ 0 , T ] \times \flattorus ; \mathbb{ R }^{ N } \right),
			\end{equation}
			\item 
			\label{here_appears_p}
			for almost every $ t \in [ 0 , T ] $ and every 
			\begin{equation*}
			 \xi \in \lp^{ p } ( [ 0 , T ] \times \flattorus; \mathbb{ R }^{ N } ) 
			\cap
			\wkp^{ 1, 2 } ( [ 0, T ] \times \flattorus ; \mathbb{ R }^{ N } ),
			\end{equation*}
			we have
			\begin{equation}
			\label{ac_weak_equation}
			\int
			\inner*{ \frac{ 1 }{ \varepsilon^{ 2 } } \nabla W ( u_{ \varepsilon} ( t ) ) }{ \xi }
			+
			\inner*{ \nabla u_{ \varepsilon } ( t ) }{ \nabla \xi } 
			+
			\inner*{\partial_{ t } u_{ \varepsilon} ( t ) }{ \xi }
			\dd{x}
			=
			0,
			\end{equation}
			\item 
			the initial conditions are attained in the sense that $ u_{ \varepsilon 
			} ( 0 ) = u_{ \varepsilon}^{ 0 } $.
		\end{enumerate}
	\end{definition}
	
	\begin{remark}
		The exponent $ p $ which appears in condition \ref{here_appears_p} is the 
		same 
		exponent as for the growth assumptions (\ref{polynomial_growth}) and 
		(\ref{polynomial_growth_derivative}) of $ W $.
		Moreover we note that given our setup, we automatically have $ \nabla W ( 
		u_{ \varepsilon } ) ( t ) \in \lp^{ p' } ( \flattorus ) $. In fact we can 
		estimate
		\begin{equation}
		\abs{ \nabla W ( u_{ \varepsilon} ) }^{ p/(p-1) }
		\lesssim
		1+ \abs{u_{ \varepsilon }}^{ p }
		\lesssim
		1 + W ( u ),
		\end{equation}
		which is integrable for almost every time $ t $ since we assume that the energy stays bounded,
		thus the integral in equation (\ref{ac_weak_equation}) is well defined.
		
		Moreover we already obtain $ 1/2 $ Hölder-continuity in time from the embedding
		\begin{equation}
		\label{w12_embeds_into_c1half}
		\wkp^{ 1 , 2} \left( [ 0 , T ] ; \lp^{ 2 } \left( \flattorus;\mathbb{ R 
		}^{ N } \right) \right)
		\hookrightarrow
		\cont^{ 1/2 } \left( [ 0 , T ] ; \lp^{ 2 } \left( \flattorus ; \mathbb{ 
			R 
		}^{ N } \right) \right),
		\end{equation}
		which follows from a generalized version of the fundamental theorem of calculus and Hölder's inequality.
	\end{remark}

	With \Cref{solution_to_ac}, we are able to state our existence result for a
	solution of the Allen--Cahn equation \cite[Lemma~2.3]{convergence_of_allen_cahn_equation_to_multiphase_mean_curvature_flow} which uses De Giorgi's minimizing movements scheme.
	Note moreover that one would expect even more regularity for a solution 
	of the Allen--Cahn equation. In fact, it has been shown by De Mottoni and 
	Schatzmann in 
	\cite{de_mottoni_schatzmann_geometrical_evolution_of_developed_interfaces} that 
	in the scalar case, a solution u is 
	smooth for positive times, at least for bounded initial data. 
	Their Ansatz is a 
	variation of parameters and the 
	regularity then follows from the smoothing effect of the heat kernel. 
	
	\begin{theorem}
		\label{existence_of_ac_solution}
		Let $ u^{ 0 } \colon \flattorus \to \mathbb{ R }^{ N } $ be such that 
		$ \energy_{ \varepsilon } ( u^{ 0 } ) < \infty $.
		Then there exists a weak solution $ u_{ \varepsilon} $ to the Allen--Cahn 
		equation (\ref{allen_cahn_eq}) in the sense of \Cref{solution_to_ac} with 
		initial data $ u^{ 0 } $.
		Furthermore the solution satisfies the energy dissipation inequality
		\begin{equation}
		\label{energy_dissipation_sharp}
		\energy_{ \varepsilon } ( u_{ \varepsilon } ( t ) )
		+
		\int_{ 0 }^{ t }
		\int
		\varepsilon \abs{ \partial_{ t } u_{ \varepsilon } }^{ 2 }
		\dd{ x }
		\dd{ s }
		\leq
		\energy_{ \varepsilon } ( u^{ 0 } )
		\end{equation}
		for every $ t \in [ 0 , T ] $. We additionally have
		$
		\partial_{ i , j }^{ 2 } u_{ \varepsilon }, \nabla W ( u_{ \varepsilon } ) \in \lp^{ 2 } ( [0, T] \times \flattorus ; \mathbb{ R }^{ N } ) 
		$
		for all $ 1 \leq i, j \leq d $. In particular we can test the weak form (\ref{ac_weak_equation}) with $ \partial_{ i , j }^{ 2 } u_{ \varepsilon } $.
	\end{theorem}

	\subsection{Convergence to an evolving partition}
	
	We now want to study of solutions to the Allen-Cahn equation (\ref{allen_cahn_eq}) as $\varepsilon$ tends to zero. 
	Let us first fix some notation. For a function $ u = \sum_{ i = 
		1 }^{ P } \mathds{ 1 }_{ \Omega_{ i } } \alpha_{ i } $ with sets of finite 
	perimeter $ \Omega_{ i } $ and the corresponding interfaces $ \Sigma_{ i j } 
	\coloneqq \partial_{ 
		\ast } \Omega_{ i } \cap \partial_{ \ast } \Omega_{ j } $, we define for a 
	given continuous function $ \varphi \in \cont ( \flattorus ) $ the 
	energy measures by
	\begin{align}
	\label{localized__epsilon_energy_multiphase}
	\energy_{ \varepsilon } ( u_{ \varepsilon } ; \varphi )
	& \coloneqq
	\int
	\varphi \left(
	\frac{ 1 }{ \varepsilon }
	W ( u_{ \varepsilon } ) 
	+
	\frac{ \varepsilon }{ 2 }
	\abs{ \nabla u_{ \varepsilon } }^{ 2 }
	\right)
	\dd{ x }
	\shortintertext{and}
	\label{localized_energy_multiphase}
	\energy ( u ; \varphi )
	& \coloneqq
	\sum_{ i < j }
	\sigma_{ i j }
	\int_{ \Sigma_{ i j } }
	\varphi
	\dd{\hm^{ d - 1 }}.
	\end{align}
	The latter is motivated through the perimeter functional introduced in 
	\Cref{section_mcf}.
	
	One of the many difficulties in the vectorial case is that there is no 
	easy choice of a primitive for $ \sqrt{ 2 W ( u ) } $ compared to the scalar 
	case. This is crucial for the classic Modica--Mortola trick to obtain $ \bv $-compactness. 
	As a suitable replacement, we consider the \emph{geodesic distance} defined as 
	\begin{align*}
	& \geodesic_distance ( u, v )
	\\
	\coloneqq{} &
	\inf
	\left\{
	\int_{ 0 }^{ 1 }
	\sqrt{ 2 W ( \gamma ) }
	\abs{ \dot{ \gamma }  }
	\dd{t}
	\,
	\colon
	\, \gamma \in \mathrm{ C }^{ 1 } \left( [0, 1 ] ; \mathbb{ R }^{ N } \right) \text{ with } \gamma( 0 ) = u,\, \gamma( 1 )= v 
	\right\}
	\end{align*}
	motivated by Baldo in \cite{baldo_minimal_interface_criterion}.
	This indeed defines a metric on $ \mathbb{ R }^{ N } $: If we have $ \geodesic_distance 
	( v, w ) = 0 $, then by the continuity of $ W $ and since it only has a 
	discrete set of zeros, we may deduce that $ v = w $. Symmetry can be seen by 
	reversing a given path between two points and the triangle inequality follows 
	from concatenation of two paths.
	
	The \emph{surface tensions} generated by $ W $  are defined as
	\begin{equation}
	\label{definition_surface_tensions}
	\sigma_{ i j } 
	\coloneqq
	\geodesic_distance ( \alpha_{ i } , \alpha_{ j } )
	\end{equation}
	and as a consequence of $ \geodesic_distance $ being a metric satisfy the 
	triangle inequality
	\begin{equation*}
	\sigma_{ i k } \leq \sigma_{ i j } + \sigma_{ j k }.
	\end{equation*}
	Furthermore $ \sigma_{ i j } $ is zero if and only if $ i $ is equal to $ j $ 
	and by symmetry, 
	$ \sigma_{ i j } $ is always the same as $ \sigma_{ j i } $.
	
	Our replacement for the primitive of $ \sqrt{ 2 W } $ is given for $ 1 \leq i \leq P $ 
	by the geodesic distance function
	\begin{equation*}
	\phi_{ i } ( u ) 
	\coloneqq
	\geodesic_distance ( \alpha_{ i } , u ).
	\end{equation*}
	Our first obstacle is the regularity of the functions $ \psi_{ \varepsilon, i } 
	\coloneqq \phi_{ i } \circ u_{ \varepsilon } $. A priori we only know that $ 
	\phi_{ i } $ is locally Lipschitz continuous on $ \mathbb{ R }^{ N } $ and thus 
	differentiable almost everywhere. If $ N = 1 $, this would already suffice 
	to deduce that $ \psi_{ \varepsilon, i  }$ is weakly differentiable.
	In higher dimensions however, $ u $ could for example move along a hypersurface 
	where $ 
	\phi_{ i } $ could in theory be nowhere differentiable since the hypersurface  
	is a Lebesgue 
	nullset. 
	This can be salvaged through the chain rule for 
	distributional derivatives by Ambrosio and Dal Maso 
	\cite[Cor.~3.2]{ambrosio_maso_chain_rule}.
	Moreover we have the crucial inequality for its gradient given by 
	\begin{equation}
		\label{bound_on_nabla_psi}
		\abs{ \nabla \psi_{ i } ( u ) }
		\leq
		\sqrt{ 2 W ( u ) }.
	\end{equation}
	We then arrive at the following convergence result, see \cite[Prop.~2.7]{convergence_of_allen_cahn_equation_to_multiphase_mean_curvature_flow}.
	
	\begin{proposition}
		\label{initial_convergence_result_multiphase}
		Let $ u_{ \varepsilon }^{ 0 } $ be well prepared initial data in the sense 
		that
		\begin{equation*}
		\lim_{ \varepsilon \to 0 }
		\energy_{ \varepsilon } ( u_{ \varepsilon }^{ 0 } ) 
		= 
		\energy ( u^{ 0 } ) 
		\eqqcolon
		\energy_{ 0 }
		< 
		\infty.
		\end{equation*}
		Then there exists for any sequence $ \varepsilon \to 0 $ some 
		non-relabelled 
		subsequence such that the solutions $ u_{ \varepsilon } $ of the 
		Allen--Cahn equation 
		(\ref{allen_cahn_eq}) with initial conditions $ u_{ \varepsilon }^{ 0 } $ 
		converge in $ \lp^{ 1 } \left( ( 0 , T ) \times \flattorus ; \mathbb{ R }^{ 
			N } \right) $ to some $ u = \sum_{ i = 1 }^{ P } \chi_{ i } \alpha_{ i } $ 
		with a partition 
		\begin{equation*} \chi \in \bv \left( ( 0 , T ) \times \flattorus ; \{ 0 , 
		1 \}^{ P } \right) .
		\end{equation*}
		Furthermore we have
		\begin{equation}
		\label{energy_estimate_for_limit_u}
		\energy ( u ( t ) ) 
		\leq
		\liminf_{ \varepsilon \to 0 }
		\energy_{ \varepsilon } ( u_{ \varepsilon } ( t ) )
		\leq 
		\energy_{ 0 }
		\end{equation}
		for almost every time $ 0 \leq t \leq T $ and $ u $ attains the initial data $ u^{ 0 } $ continuously in $  \lp^{ 2 } \left( \flattorus ; \mathbb{ R }^{ d } \right) $.
		Moreover for all $ 1 \leq i \leq P $, the compositions $ \psi_{ \varepsilon , i } \coloneqq \phi_{ i } \circ 
		u_{ 
			\varepsilon } $ are uniformly bounded in $ \bv ( ( 0 , T ) \times 
		\flattorus ) $ and converge to $ \psi_{ i } \coloneqq \phi_{ i } \circ u $ in $ \lp^{ 1 } \left( 
		( 0 , T ) \times \flattorus \right) $. 
	\end{proposition}

	\section{De Giorgi's mean curvature flow}
	\label{chapter_de_giorgis_mcf}
	
	In this chapter, we build on the results of \cite{convergence_of_allen_cahn_equation_to_multiphase_mean_curvature_flow}, but introduce 
	a different solution concept for mean curvature flow, namely a De Giorgi type $ 
	\bv 
	$-solution to mean curvature flow. A similar solution concept has been 
	introduced in 
	\cite{laux_otto_convergence_of_thresholding_scheme_for_mmcf} and
	\cite[Def.~1]{laux_lelmi_de_giorgis_inequality_for_the_threshholding_scheme}, 
	but in the context of convergence of the thresholding scheme to mean curvature 
	flow. Moreover, we will also compare it to the solution concept 
	\cite[Def.~1]{hensel_laux_varifold_solution_concept_for_mean_curvature_flow}, 
	which permits oriented varifolds to be solutions to mean curvature flow. 
	This provides a more general notion of solution and is of use when the 
	assumption of energy convergence falls away. See also the later discussion in \Cref{subsection_de_giorgi_type_varifold_solutions_for_mcf}.

	\subsection{Conditional convergence to De Giorgi's mean curvature flow}
	
	In this section, we shall state our solution concept and prove convergence to 
	the aforementioned.
	
	\begin{definition}[De Giorgi type $ \bv $-solution to multiphase mean 
		curvature 
		flow]
		\label{de_giorgi_solution_to_mmcf}
		Fix some finite time horizon $ T < \infty $, a $ (P \times P) $-matrix of 
		surface tensions $ \sigma $ and initial data $ \chi^{ 0 } \colon \flattorus 
		\to \{ 0 , 1 \}^{ P } $ with $ \energy_{ 0 } \coloneqq \energy ( \chi^{ 0 } 
		) < \infty $ and 
		$ \sum_{ i = 1 }^{ P } \chi_{ i }^{ 0 } = 1 $ almost everywhere. We say that
		\begin{equation*}
		\chi \in \cont \left(
		[ 0 , T ]
		;
		\lp^{ 2 } \left( \flattorus ; \{ 0 , 1 \}^{ P } \right)
		\right)
		\cap 
		\bv \left(
		( 0 , T ) \times \flattorus ; \{ 0 , 1 \}^{ P } 
		\right)
		\end{equation*}
		with $ \esssup_{ 0 \leq t \leq T } \energy ( \chi ) < \infty  $ and $ 
		\sum_{ i = 1 
		}^{ P } \chi_{ i } = \sum_{ i = 1 }^{ P } \mathds{ 1 }_{ \Omega_{ i } } = 
		1  $ almost everywhere is a \emph{De Giorgi type }$\bv$\emph{-solution to 
			multiphase mean 
			curvature flow with initial data} $ \chi^{ 0 } $ \emph{and surface 
			tensions} $ \sigma $ if the following holds. 
		\begin{enumerate}
			\item 
			For all $ 1 \leq i \leq P $, there exists a normal 
			velocity $ V_{ i } \in \lp^{ 2 } ( \abs{ \nabla \chi_{ i } } \dd{ t } 
			) $ 
			such that
			\begin{equation*}
			\partial_{ t } \chi_{ i }
			=
			V_{ i } \abs{ \nabla \chi_{ i } } \dd{ t }
			\end{equation*}
			holds in the distributional sense on $ ( 0 , T ) \times \flattorus $.
			
			\item 
			There exist a mean curvature vector $ H \in 
			\lp^{ 2 } ( \energy ( u ; \cdot) \dd{ t } ; \mathbb{ R }^{ d } ) $ 
			which satisfies
			\begin{align}
			\notag
			& 
			\sum_{ 1 \leq i < j \leq P }
			\sigma_{ i j }
			\int_{ 0 }^{ T }
			\int_{ \Sigma_{ i j } }
			\inner*{ H }{ \xi }
			\dd{ \hm^{ d - 1 } }
			\dd{ t }
			\\
			\label{mean_curvature_vector_bv_de_giorgi}
			={} &
			-
			\sum_{ 1 \leq i < j \leq P }
			\sigma_{ i j }
			\int_{ 0 }^{ T }
			\int_{ \Sigma_{ i j } }
			\inner*{
				\diff \xi }
			{ \mathrm{Id} - \nu_{ i } \otimes \nu_{ i } }
			\dd{ \hm^{ d - 1 } }
			\dd{ t }
			\end{align}
			for all test vector fields 
			$ \xi \in \cont_{ \mathrm{c} }^{ \infty } \left(
			( 0 , T ) \times \flattorus ; \mathbb{ R }^{ d }
			\right) $,
			where $ \nu_{ i } \coloneqq \nabla \chi_{ i } / \abs{ \nabla \chi_{ i } 
			} $ are the inner unit normals and $ \Sigma_{ i j } \coloneqq 
			\partial_{ \ast } \Omega_{ i } \cap \partial_{ \ast } \Omega_{ j } $
			is the $ (i, j )$-th interface.
			
			\item 
			The partition $ \chi $ satisfies a De Giorgi type optimal energy 
			dissipation inequality in the sense that for almost every time $ 0 < T' 
			< T $, we have
			\begin{equation}
			\label{optimal_energy_dissipation_solution}
			\energy ( \chi ( T' ) )
			+
			\frac{ 1 }{ 2 }
			\sum_{ 1 \leq i < j \leq P }
			\sigma_{ i j }
			\int_{ 0 }^{ T' }
			\int_{ \Sigma_{ i j } }
			V_{ i }^{ 2 }
			+
			\abs{ H }^{ 2 }
			\dd{ \hm^{ d - 1 } }
			\dd{ t }
			\leq
			\energy_{ 0 }.
			\end{equation}
			
			\item
			The initial data is attained in $ \cont \left( [ 0 , T ] ; 
			\lp^{ 2 } ( \flattorus ) \right) $.
		\end{enumerate}
	\end{definition}
	
	\subsection{Localization estimates}
	\label{section_localization_estimates}
	
	Indeed, we obtain the existence of such solutions through a phase-field approximation with the Allen-Cahn equation as long as we assume the convergence of the integrated energies given by
	\begin{equation}
	\label{energy_convergence}
		\lim_{ \varepsilon \to 0 }
		\int_{ 0 }^{ T }
			\energy_{ \varepsilon } ( u_{ \varepsilon } )
		\dd{ t }
		=
		\int_{ 0 }^{ T }
			\energy ( u )
		\dd{ t }.
	\end{equation}
	
	A consequence of this is the \emph{equipartition of energies}, already observed by Ilmanen in \cite{ilmanen_convergence_of_ac_to_brakkes_mcf} for the scalar case under weaker assumptions using a comparison principle. The proof for the multiphase case can be found in 
	\cite[Lemma~2.11]{convergence_of_allen_cahn_equation_to_multiphase_mean_curvature_flow}.
	
	\begin{lemma}
		\label{equipartition_of_energies}
		In the situation of \Cref{initial_convergence_result_multiphase} and under the energy 
		convergence assumption (\ref{energy_convergence}), we have 
		for any continuous function $ \varphi \in \cont ( \flattorus ) $ 
		that
		\begin{align*}
		\energy ( u ; \varphi )
		=
		\lim_{ \varepsilon \to 0 }
		\energy_{ \varepsilon } ( u_{ \varepsilon } ; \varphi )
		& = 
		\lim_{ \varepsilon \to 0 }
		\int
		\varphi
		\sqrt{ 2 W ( u_{ \varepsilon } ) }
		\abs{ \nabla u_{ \varepsilon } }
		\dd{ x }
		\\
		& =
		\lim_{ \varepsilon \to 0 }
		\int
		\varphi
		\varepsilon
		\abs{ \nabla u_{ \varepsilon } }^{ 2 }
		\dd{ x}
		\\
		& =
		\lim_{ \varepsilon \to 0 }
		\int
		\varphi
		\frac{ 1 }{ \varepsilon }
		2 W ( u_{ \varepsilon } )
		\dd{ x }
		\end{align*}
		for almost every time $ 0 \leq t \leq T $.
	\end{lemma}

In order to prove convergence, we want to 
reduce the multiphase case to the two-phase case. The central idea here is to 
cover the flat torus with a suitable collection of balls. Then we argue that, 
up to a small error, we can choose 
for each ball a pair of majority phases $ (i, j) $ such that the partition looks like a 
two-phase mean curvature flow on the ball, see 
\Cref{figure_localization_of_mmcf}.

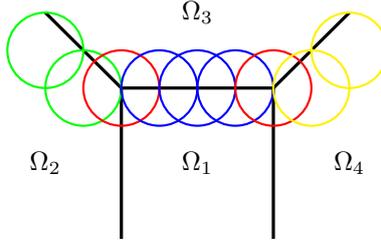
\begin{figure}[h]
	\centering
	\begin{tikzpicture}
	
	\draw[very thick]  (1,-1) -- (2, -2 );
	\draw[very thick]  (4,-2) --(5,-1);
	
	\draw[very thick](2,-2)--(4,-2);
	\draw[very thick] (4,-2)--(4,-4);
	\draw[very thick] (2, -4) -- (2, -2);

	\draw[green,thick] (1,-1.5) circle (0.5cm);	
	\draw[green,thick] (1.5,-2) circle (0.5cm);	
	\draw[red,thick] (2,-2) circle (0.5cm);	
	\draw[blue,thick] (2.5,-2) circle (0.5cm);	
	\draw[blue,thick] (3,-2) circle (0.5cm);	
	\draw[blue,thick] (3.5,-2) circle (0.5cm);	
	\draw[red,thick] (4,-2) circle (0.5cm);	
	\draw[yellow,thick] (4.5,-2) circle (0.5cm);	
	\draw[yellow,thick] (5,-1.5) circle (0.5cm);	
	
	\node at (3,-3) {$\Omega_{ 1 }$};
	\node at (1,-3) {$\Omega_{2}$};
	\node at (3,-1) {$\Omega_{3}$};
	\node at (5,-3) {$\Omega_{4}$};
	\end{tikzpicture}
	\caption{Localization of multiphase mean curvature flow. For the green 
		balls, we choose the majority phase (2,3), for the blue ones (1,3) and for 
		the yellow balls (3,4). The red balls however give us an error.}
	
	\label{figure_localization_of_mmcf}
\end{figure}

To formulate this rigorously let $ r > 0 $ and define the covering $ \mathcal{ 
	B }_{ r 
} $ of the flat torus by
\begin{equation*}
\mathcal{B}_{ r } \coloneqq
\left\{
B_{ r } ( c ) 
\, \colon \,
c \in \mathcal{ L }_{ r }
\right\},
\end{equation*}
where the set of centers $ \mathcal{ L }_{ r } $ is given by $ \mathcal{ L }_{ 
	r } \coloneqq \flattorus \cap (r/\sqrt{ d }) \mathbb{ Z }^{ d } $. Moreover 
let 
$ \rho_{ B } $ be a smooth cutoff for the ball $ B $ with support in the ball 
with the same center, but double the radius.

Additionally to choosing a majority phase, we can also argue that along the 
chosen majority phase, we have a local flatness. Thus we may approximate the 
inner unit normal up to an arbitrarily small error by a constant unit vector. 
This is captured by the following lemma found in 
\cite{laux_otto_convergence_of_thresholding_scheme_for_mmcf}. The proof is 
based on De Giorgi's structure theorem for sets of finite perimeter.

\begin{lemma}
	\label{localization_lemma_with_normals}
	For every $ \delta > 0 $ and every partition $ \chi \colon \flattorus \to 
	\{ 0 , 1 \}^{ P } $ such that $ \chi_{ i } \in \bv ( \flattorus ) $ holds 
	for all $ 1 \leq i \leq P $, there exist some $ r_{ 0 } > 0 $ such that for 
	all $ 0 < r < r_{ 0 } $, we find for every ball $ B \in \mathcal{B}_{ r } $ 
	some unit vector $ \nu_{ B } $ such that
	\begin{align*}
	& \sum_{ B \in \mathcal{B}_{ r } }
	\min_{ i \neq j }
	\int
	\rho_{ B }
	\abs{ \nu_{ i } - \nu_{ B } }^{ 2 }
	\abs{ \nabla \chi_{ i } }
	+
	\int
	\rho_{ B }
	\abs{ \nu_{ j } + \nu_{ B } }^{ 2 }
	\abs{ \nabla \chi_{ j } }
	+
	\sum_{ k \notin \{ i , j \} }
	\int 
	\rho_{ B }
	\abs{ \nabla \chi_{ k } }
	\\
	\lesssim{} &
	\delta \energy ( \chi ).
	\end{align*}
\end{lemma}

Note that summing  $ 
\abs{ \nabla \chi_{ k } } $ over all $ k \notin \{ i , j \} $ 
corresponds to summing over all interfaces which are not the $ (i,j)$-th interface, see also \cite{baldo_minimal_interface_criterion}.
Moreover the second summand is redundant. In fact the last 
summand provides us with a localization on the $ (i,j)$-th interface, on which 
we have $ \nu_{ i } = - \nu_{ j } $.
However it is convenient to keep it so that we do not have to repeat this 
argument.

\begin{remark}
	\label{localization_estimate_weaker}
	This localization estimate also implies the smallness of
	\begin{equation}
	\label{alternative_localization_equality}
	\sum_{ B \in \mathcal{ B }_{ r } }
	\min_{ i }
	\abs{
		\energy ( \chi; \rho_{ B } )
		-
		\int
		\rho_{ B }
		\abs{ \nabla \psi_{ i } }
	}
	=
	\sum_{ B \in \mathcal{ B }_{ r } }
	\energy ( \chi ; \rho_{ B } )
	-
	\max_{ i }
	\int
	\rho_{ B }
	\abs{ \nabla \psi_{ i } }
	\end{equation}
	in the same sense as in the above \Cref{localization_lemma_with_normals}. 
	Notice that the equality (\ref{alternative_localization_equality}) follows 
	from 
	\begin{equation}
	\label{rewriting_variation_of_psi_i}
	 \abs{ 
		\nabla 
		\psi_{ i } }
	=
	\sum_{ 1 \leq k < l \leq P}
		\abs{ \sigma_{ i k } - \sigma_{ i l } }
		\hm^{ d - 1 } \llcorner_{ \Sigma_{ k l } }  
	\leq \energy ( \chi ; \cdot ),
	\end{equation} 
	where we remind that $ \psi_{ i } \coloneqq \phi_{ i } \circ u $.
	The smallness follows 
	since we have for every $ i \neq j $ that
	\begin{align*}
	&
	\energy ( \chi ; \rho_{ B } ) 
	-
	\int
	\rho_{ B }
	\abs{ \nabla \psi_{ i } }
	\\
	={} &
	\sum_{ 1 \leq k < l \leq P }
	\sigma_{ k l }
	\int_{ \Sigma_{ k l } }
	\rho_{ B }
	\dd{ \hm^{ d -1 } }
	-
	\sum_{ 1 \leq k < l \leq P }
	\abs{ \sigma_{ i k } - \sigma_{ i l } }
	\int_{ \Sigma_{ k l } }
	\rho_{ B }
	\dd{ \hm^{ d - 1 } }
	\\
	={} &
	\sum_{ \substack{1 \leq k < l \leq P \\ ( k, l ) \neq ( i, j ) } }
	\left(
	\sigma_{ k l } - \abs{ \sigma_{ i k } - \sigma_{ i l } }
	\right)
	\int_{ \Sigma_{ k l } }
	\rho_{ B }
	\dd{ \hm^{ d - 1 } }
	\\
	\lesssim{} &
	\sum_{ k \notin \{ i, j \} }
	\int
	\rho_{ B }
	\abs{ \nabla \chi_{ k } }.
	\end{align*}
	Thus we can estimate the error by the last summand of the error in 
	\Cref{localization_lemma_with_normals}.
\end{remark}

	\subsection{Conditional convergence of Allen--Cahn equations to De Giorgi's multiphase mean curvature flow}
	
	We are now in the position to prove the desired conditional convergence result.
	\begin{theorem}
		\label{convergence_to_de_giorgis_multiphase_mcf}
		Let $ W \colon \mathbb{ R }^{ N } \to [ 0, 
		\infty ) $ be a smooth multiwell potential satisfying the assumptions 
		(\ref{polynomial_growth})-(\ref{perturbation bound}). Let $ T < \infty 
		$ be an arbitrary finite time horizon. Let 
		$ u_{ \varepsilon }^{ 0 } \colon \flattorus \to \mathbb{ R }^{ N } $ be a 
		sequence of initial data
		approximating a partition 
		$ \chi^{ 0 } \in \bv \left( \flattorus ; \{ 0 , 1 \}^{ P } \right) $ 
		in the sense that 
		$ u_{ \varepsilon }^{ 0 } \to u^{ 0 } =  \sum_{ 1 \leq i \leq P } 
		\chi_{ i }^{ 0 } \alpha_{ i } $ 
		holds pointwise almost everywhere and 
		\begin{equation*} 
		\energy_{ 0 } 
		\coloneqq 
		\energy ( \chi^{ 0 } ) 
		= 
		\lim_{ \varepsilon \to 0 } 
		\energy_{ \varepsilon } ( u_{ \varepsilon }^{ 0 } ) 
		< 
		\infty.
		\end{equation*}
		Then we have for 
		some subsequence of solutions $ u_{\varepsilon } $ to the Allen--Cahn 
		equation
		(\ref{allen_cahn_eq})  with initial datum $ u_{ 
			\varepsilon }^{ 0 } $ that there exists a time-dependent partition $ 
		\chi $ 
		with 
		\begin{equation*}
		\chi \in \bv \left( ( 0 , T ) \times \flattorus ; \{ 0 , 1 \}^{ P } 
		\right) 
		\end{equation*} and
		$ \chi 
		\in \cont \left( [ 0 , T ] ; \lp^{ 2 } \left( \flattorus ;  \{ 0 , 1 
		\}^{ P } \right) \right) $ such that $ u_{ \varepsilon } $ converges to 
		$ u \coloneqq \sum_{ 1 \leq i \leq P } \chi_{ i } \alpha_{ i } $ in $ \lp^{ 1 } \left( ( 0 , T ) \times \flattorus ; \mathbb{ R }^{ N } \right)$. 
		Moreover $ u $ attains the initial data $ u^{ 0 } $ continuously in $ \lp^{ 2 }\left( \flattorus ; \mathbb{ R }^{ N }  \right) $. 
		
		If we 
		additionally assume that the 
		time-integrated energies converge (\ref{energy_convergence}), then $ 
		\chi $ is a De Giorgi type $ \bv $-solution to multiphase mean curvature 
		flow in the sense 
		of 
		\Cref{de_giorgi_solution_to_mmcf}.
	\end{theorem} 
	
	\begin{proof}
		\Cref{initial_convergence_result_multiphase} already provides us with all claims except the last.
		The idea of the proof is that we already have an optimal energy dissipation 
		inequality for the Allen--Cahn equation given by 
		(\ref{energy_dissipation_sharp}).
		If we additionally use the Allen--Cahn equation once, we arrive at 
		the De Giorgi type optimal energy dissipation inequality given 
		by
		\begin{equation*}
		\label{energy_dissipation_sharp_with_curvature}
		\energy_{ \varepsilon } ( u_{ \varepsilon } ( T' ) )
		+
		\frac{ 1 }{ 2 }
		\int_{ 0 }^{ T' }
		\int
		\varepsilon \abs{ \partial_{ t } u_{ \varepsilon } }^{ 2 }
		+
		\frac{ 1 }{ \varepsilon }
		\abs{
			\varepsilon \Delta u_{ \varepsilon } 
			-
			\frac{ 1 }{ \varepsilon } \nabla W ( u_{ \varepsilon } )
		}^{ 2 }
		\dd{ x }
		\dd{ t }
		\leq
		\energy_{ \varepsilon } ( u_{ \varepsilon}^{ 0 } ).
		\end{equation*}
		Our hope is that as $ \varepsilon $ tends to zero, we can pass to the 
		optimal energy dissipation inequality 
		(\ref{optimal_energy_dissipation_solution}) for $ \chi $ through lower 
		semicontinuity. 
		Since we assume the convergence of the initial energies and energy 
		convergence for almost every time, the only terms we have to care about are
		the velocity and curvature term. 
		The desired lower semicontinuity of the velocity term reads
		\begin{equation}
		\label{lsc_of_velocity}
		\liminf_{ \varepsilon \to 0 }
		\frac{ 1 }{ 2 }
		\int_{ 0 }^{ T' }
		\int
		\varepsilon 
		\abs{ \partial_{ t } u_{ \varepsilon } }^{ 2 }
		\dd{ x }
		\dd{ t }
		\geq
		\frac{ 1 }{ 2 }
		\sum_{ i < j  }
		\sigma_{ i j }
		\int_{ 0 }^{ T' }
		\int_{ \Sigma_{ i j } }
		V_{ i }^{ 2 }
		\dd{ \hm^{ d - 1 } }
		\dd{ t }
		\end{equation}
		and the lower semicontinuity of the curvature term is given by
		\begin{align}
		\label{lsc_of_curvature}
		\notag
		& \liminf_{ \varepsilon \to 0 }
		\frac{ 1 }{ 2 }
		\int_{ 0 }^{ T' }
		\int
		\frac{ 1 }{ \varepsilon }
		\abs{
			\varepsilon
			\Delta u_{ \varepsilon }
			-
			\frac{ 1 }{ \varepsilon }
			\nabla W ( u_{ \varepsilon } ) 
		}^{ 2 }
		\dd{ x }
		\dd{ t }
		\\
		\geq{} &
		\frac{ 1 }{ 2 }
		\sum_{ i < j  }
		\sigma_{ i j }
		\int_{ 0 }^{ T' }
		\int_{ \Sigma_{ i j } }
		\abs{ H }^{ 2 }
		\dd{ \hm^{ d - 1 } }
		\dd{ t }.
		\end{align}
		The existence of square integrable distributional normal velocities has been proven in \cite[Prop.~2.10]{convergence_of_allen_cahn_equation_to_multiphase_mean_curvature_flow}.
		Moreover we have to show the existence of the mean curvature vector $ H $.
		We could cheat in this step and simply use that by the main result
		\cite[Thm.~1.2]{convergence_of_allen_cahn_equation_to_multiphase_mean_curvature_flow}, 
		we already know that the tangential 
		divergence applied to $ \xi $ is given by the velocity. In other words, 
		that we already have $ V_{ i } \nu_{ i } = -H $ on $ \Sigma_{ i j } $. But 
		we want 
		to present a more direct approach.
		
		Consider the linear functional 
		\begin{equation*}
		L ( \xi )
		\coloneqq
		- \sum_{ 1 \leq i < j \leq P }
		\sigma_{ i j }
		\int_{ 0 }^{ T }
		\int_{ \Sigma_{ i j } }
		\inner*{ \diff \xi }{ \mathrm{Id} - \nu_{ i } \otimes \nu_{ 
				i } }
		\dd{ \hm^{ d - 1 } }
		\dd{ t }
		\end{equation*}
		defined on test vector fields $ \xi $. Then $ L $ is bounded with respect 
		to the $ \lp^{ 2 } $-norm on $ ( 0 , T ) \times \flattorus $ equipped with the measure
		$ \energy ( u ; \cdot ) \dd{ t } $ since by the convergence of the 
		curvature 
		term observed in 
		\cite[Prop.~3.1]{convergence_of_allen_cahn_equation_to_multiphase_mean_curvature_flow}, we have
		\begin{align*}
		& \abs{ L ( \xi ) }
		\\
		={} &
		\lim_{ \varepsilon \to 0 }
		\abs{ 
			-
			\int_{ 0 }^{ T }
			\int
			\inner*{
				\varepsilon \Delta u_{ \varepsilon } 
				-
				\frac{ 1 }{ \varepsilon }
				\nabla W ( u_{ \varepsilon } )
			}
			{ \diff u_{ \varepsilon } \xi }
			\dd{ x }
			\dd{ t }
		}
		\\
		\leq{} &
		\liminf_{ \varepsilon \to 0 }
		\left(
		\int_{ 0 }^{ T }
		\int
		\frac{ 1 }{ \varepsilon }
		\abs{ 
			\varepsilon \Delta u_{ \varepsilon }
			-
			\frac{ 1 }{ \varepsilon }
			\nabla W ( u_{ \varepsilon } )
		}^{ 2 }
		\dd{ x }
		\dd{ t }
		\right)^{ 1/2 }
		\left(
		\int_{ 0 }^{ T }
		\int
		\varepsilon
		\abs{ \diff u_{ \varepsilon } \xi }^{ 2 }
		\dd{ x }
		\dd{ t }
		\right)^{ 1/2 }
		\\
		\leq{} &
		\liminf_{ \varepsilon \to 0 }
		\left(
		\int_{ 0 }^{ T }
		\int
		\varepsilon
		\abs{ 
			\partial_{ t } u_{ \varepsilon }
		}^{ 2 }
		\dd{ x }
		\dd{ t }
		\right)^{ 1/2 }
		\left(
		\int_{ 0 }^{ T }
		\int
		\varepsilon
		\abs{ \nabla u_{ \varepsilon } }^{ 2 }
		\abs{ \xi }^{ 2 }
		\dd{ x }
		\dd{ t }
		\right)^{ 1/2 }.
		\end{align*}
		The first factor stays uniformly bounded due to the energy dissipation 
		inequality (\ref{energy_dissipation_sharp}), and by the equipartition of 
		energies (\Cref{equipartition_of_energies}), the second factor 
		converges to the $ \lp^{ 2 } $-norm of $ \xi $ with respect to the energy 
		measure, proving our claim. Therefore we can extend the functional to 
		the square integrable functions with respect to the energy measure. By 
		Riesz representation theorem we obtain the existence of the desired mean 
		curvature vector $ H $.
		
		Let us now consider the lower semicontinuity of the curvature 
		term, which in particular yields a sharp version of the previous estimate. 
		Let again $ \xi $ be some test vector field. Then for all $ 
		\varepsilon > 0 $ and some fixed time, we have by Young's inequality 
		\begin{align*}
		& 
		\liminf_{ \varepsilon \to 0 }
		\frac{ 1 }{ 2 }
		\int
		\frac{ 1 }{ \varepsilon }
		\abs{ 
			\varepsilon
			\Delta u_{ \varepsilon }
			-
			\frac{ 1 }{ \varepsilon }
			\nabla W ( u_{ \varepsilon } )
		}^{ 2 }
		\dd{ x }
		\\
		\geq{} &
		\liminf_{ \varepsilon \to 0 }
		\int
		\inner*{ 
			\varepsilon
			\Delta u_{ \varepsilon }
			-
			\frac{ 1 }{ \varepsilon }
			\nabla W ( u_{ \varepsilon } )
		}{
			\diff u_{ \varepsilon } \xi
		}
		\dd{ x }
		-
		\frac{ 1 }{ 2 }
		\int 
		\varepsilon
		\abs{ \diff u_{ \varepsilon } \xi }^{ 2 }
		\dd{ x }
		\\
		\geq{} &
		\liminf_{ \varepsilon \to 0 }
		\int
		\inner*{ 
			\varepsilon
			\Delta u_{ \varepsilon }
			-
			\frac{ 1 }{ \varepsilon }
			\nabla W ( u_{ \varepsilon } )
		}{
			\diff u_{ \varepsilon } \xi
		}
		\dd{ x }
		-
		\limsup_{ \varepsilon \to 0 }
		\frac{ 1 }{ 2 }
		\int
			\varepsilon 
			\abs{ \diff u_{ \varepsilon } }^{ 2 }
			\abs{ \xi }^{ 2 }
		\\
		\geq{} &
		-\energy \left( \chi ; \inner*{ H }{ \xi } \right)
		-
		\frac{ 1 }{ 2 }
		\energy \left( \chi ; \abs{ \xi }^{ 2 } \right).
		\end{align*}
		Here the last inequality is due to the above mentioned convergence of the curvature term and the equipartition of energies \Cref{equipartition_of_energies}.
		Since this inequality holds for any test vector field, we may take a 
		sequence of test vector fields $ \xi_{ n } $ satisfying
		\begin{equation*}
		\lim_{ n \to \infty }
		\norm{ \xi_{ n } + H }_{ \lp^{ 2 } \left( \flattorus , \energy 
			( 
			\chi ; \cdot ) ; \mathbb{ R }^{ d } \right) } 
		=
		0.
		\end{equation*}
		This then yields the desired inequality (\ref{lsc_of_curvature}) by 
		applying Fatou's lemma.
		
		In principle the proof is now already done, since the main result by Laux and Simon already gives us that for almost every 
		time $ t $, we have $ V_{ i } \nu_{ i } = - H $ on $ \Sigma_{ i , j 
		} $ $ \hm^{ d-  1 } $-almost everywhere. We instead want to present another proof which 
		directly proves the lower semicontinuity of the velocity term and may leave 
		more room for future generalizations.
		
		Let us first consider the two-phase case $ N = 1 $ and $ P = 2 $. By a 
		similar duality argument as for the 
		lower semicontinuity of the curvature term, we compute 
		that for 
		every test function $ \varphi $, we have
		\begin{align*}
		& \liminf_{ \varepsilon \to 0 }
		\frac{ 1 }{ 2 }
		\int_{ 0 }^{ T }
		\int
		\varepsilon \abs{ \partial_{ t } u_{ \varepsilon } }^{ 2 }
		\dd{ x }
		\dd{ t }
		\\
		\geq{} &
		\liminf_{ \varepsilon \to 0 }
		\int_{ 0 }^{ T }
		\int
		\partial_{ t } u_{ \varepsilon } 
		\phi ' ( u_{ \varepsilon } )
		\varphi
		\dd{ x }
		\dd{ t }
		-
		\frac{ 1 }{ 2 }
		\int_{ 0 }^{ T }
		\int
		\frac{ 1 }{ \varepsilon }
		\left( \phi ' ( u_{ \varepsilon } ) \varphi \right)^{ 2 }
		\dd{ x }
		\dd{ t }
		\\
		={} &
		\liminf_{ \varepsilon \to 0 }
		\int_{ 0 }^{ T }
		\int
		\partial_{ t } \psi_{ \varepsilon }
		\varphi
		\dd{ x }
		\dd{ t }
		-
		\frac{ 1 }{ 2 }
		\int_{ 0 }^{ T }
		\int
		\frac{ 1 }{ \varepsilon }
		2 W ( u_{ \varepsilon } )
		\varphi^{ 2 }
		\dd{ x }
		\dd{ t }
		\\
		={} &
		\sigma
		\int_{ 0 }^{ T }
		\int_{ \Sigma }
		\varphi V
		\dd{ \hm^{ d - 1 } }
		\dd{ t }
		-
		\frac{ 1 }{ 2 }
		\sigma
		\int_{ 0 }^{ T }
		\int_{ \Sigma }
		\varphi^{ 2 }
		\dd{ \hm^{ d - 1 } }
		\dd{ t }.
		\end{align*}
		Here the last equality uses the weak convergence of $ \partial_{ t } (\phi \circ u_{ \varepsilon }) = \partial_{  t } 
		\psi_{ \varepsilon } $ to $ \partial_{ t } (\phi \circ u) = \partial_{  t } \psi = \sigma V \abs{ \nabla 
			\chi } \dd{ t } $ given by \Cref{initial_convergence_result_multiphase} for the first summand 
		and the equipartition of energies (\Cref{equipartition_of_energies}) for 
		the second summand.
		Since the inequality holds for any test function $ \varphi $, we may 
		plug in a sequence of test functions $ \varphi_{ n } $ satisfying
		\begin{equation*}
		\lim_{ n \to \infty }
		\norm{ \varphi_{ n } - V }_{ \lp^{ 2 } \left( ( 0 , T ) \times 
			\flattorus , \hm^{ d - 1 }  \llcorner_{ \Sigma } \dd{ t } \right) }
		= 0
		\end{equation*}
		and thereby obtain the desired inequality (\ref{lsc_of_velocity}).
		
		For the multiphase case, we do not find an immediate generalization of this 
		proof, but rather have to work with a localization argument in 
		order to 
		obtain a reduction to the two-phase case.
		
		For a localization in time, let $ \delta > 0 $ 
		and $ 0 = T_{ 0 } < T_{ 1 } < \dotsc < T_{ K } = T' $ be a partition of $ [ 
		0 , T' ] $. Let $ ( g_{ k } )_{ k = 1 , \dotsc , K } $ be a partition of 
		unity with respect to the intervals $ \left( ( T_{ k - 1 } - \delta , T_{ k 
		} + \delta )\right)_{ 1 \leq k \leq K} $, let $ r > 0 $
		and $ \eta_{ B } $ a partition of unity in space as in 
		\Cref{localization_lemma_with_normals}. Moreover we fix some $ R > 0 $. 
		Then we estimate by Young's inequality
		\begin{align*}
		& \liminf_{ \varepsilon \to 0 }
		\frac{ 1 }{ 2 }
		\int_{ 0 }^{ T }
		\int
		\varepsilon
		\abs{ \partial_{ t } u_{ \varepsilon } }^{ 2 }
		\dd{ x }
		\dd{ t }
		\\
		={} &
		\liminf_{ \varepsilon \to 0 }
		\sum_{ k = 1 }^{ K }
		\sum_{ B \in \mathcal{ B }_{ r } } 
		\frac{ 1 }{ 2 }
		\int_{ 0 }^{ T }
		\int
		g_{ k } \eta_{ B }
		\varepsilon
		\abs{ \partial_{ t } u_{ \varepsilon } }^{ 2 }
		\dd{ x }
		\dd{ t }
		\\
		\geq{} &
		\liminf_{ \varepsilon \to 0 }
		\sum_{ k = 1 }^{ K }
		\sum_{ B \in \mathcal{ B }_{ r } }
		\max_{ 1 \leq i \leq P }
		\sup_{ 
			\substack{ 
				\varphi \in \cont_{ \mathrm{c} }^{ \infty } 
				\left( ( 0 , T ) \times \flattorus \right)
				\\
				\abs{ \varphi } \leq R  
			}
		}
		\int_{ 0 }^{ T }
		\int
		g_{ k } \eta_{ B }
		\inner*{ \nabla \phi_{ i } ( u_{ 
				\varepsilon 
			} ) }{ \partial_{ t } u_{ \varepsilon } }
		\varphi
		\dd{ x }
		\dd{ t } 
		\\
		& \qquad \qquad \qquad \qquad \qquad \qquad \qquad 
		\quad 
		-
		\frac{ 1 }{ 2 }
		\int_{ 0 }^{ T }
		\int
		g_{ k } \eta_{ B }
		\frac{ 1 }{ \varepsilon }
		\abs{ \nabla \phi_{ i } ( u_{ \varepsilon } 
			) }^{ 2 }
		\varphi^{ 2 }
		\dd{ x }
		\dd{ t }.
		\end{align*}
		We identify via the chain rule that $ \inner*{ \nabla \phi_{ i } ( 
			u_{ \varepsilon } ) }{ \partial_{ t } u_{ \varepsilon } } = \partial_{ t } 
		\psi_{ \varepsilon , i } $ and moreover remember $ \abs{ \nabla \phi_{ i } 
		} \leq \sqrt{ 2 W } $ from inequality (\ref{bound_on_nabla_psi}).
		We pull the limit inferior inside the double sum and the suprema and thus 
		obtain 
		that this term can be estimated from below by
		\begin{align*}
		&
		\sum_{ k = 1 }^{ K }
		\sum_{ B \in \mathcal{ B }_{ r } }
		\max_{ 1 \leq i \leq P }
		\sup_{ 
			\substack{ 
				\varphi \in \cont_{ \mathrm{c} }^{ \infty } 
				\left( ( 0 , T ) \times \flattorus \right)
				\\
				\abs{ \varphi } \leq R  
			}
		}
		\liminf_{ \varepsilon \to 0 }
		\int_{ 0 }^{ T }
		\int
		g_{ k } \eta_{ B }
		\varphi
		\partial_{ t } \psi_{ \varepsilon, i }
		\\
		&\qquad \qquad \qquad \qquad \qquad \quad \;\;
		-
		\limsup_{ \varepsilon \to 0 }
		\frac{ 1 }{ 2 }
		\int_{ 0 }^{ T } 
		\int
		g_{ k }
		\eta_{ B }
		\frac{ 2 }{ \varepsilon }
		W ( u_{ \varepsilon } )
		\varphi^{ 2 }
		\dd{ x }
		\dd{ t }
		\\
		={}&
		\sum_{ k = 1 }^{ K }
		\sum_{ B \in \mathcal{ B }_{ r } }
		\max_{ 1 \leq i \leq P }
		\sup_{ 
			\substack{ 
				\varphi \in \cont_{ \mathrm{c} }^{ \infty } 
				\left( ( 0 , T ) \times \flattorus \right)
				\\
				\abs{ \varphi } \leq R  
			}
		}
		\int_{ 0 }^{ T }
		\int
		g_{ k } \eta_{ B }
		\varphi
		\partial_{ t } \psi_{ i }
		-
		\frac{ 1 }{ 2 }
		\int_{ 0 }^{ T }
		\energy \left( \chi ; g_{ k } \eta_{ B } \varphi^{ 2 } \right)
		\dd{ t }.
		\end{align*}
		For the equality we used the equipartition of energies \Cref{equipartition_of_energies} and the $ \lp^{ 1 } $-convergence of $ \psi_{ \varepsilon , i } $ to $ \psi_{ i } $ given by \Cref{initial_convergence_result_multiphase}. 
		By adding zero, we get
		\begin{align*}
		& \int_{ 0 }^{ T }
		\int
		g_{ k } \eta_{ B }
		\varphi
		\partial_{ t } \psi_{ i }
		\dd{ x }
		\dd{ t } 
		-
		\frac{ 1 }{ 2 }
		\int_{ 0 }^{ T }
		\energy \left( \chi ; g_{ k } \eta_{ B } \varphi^{ 2 } \right)
		\dd{ t }
		\\
		={}&
		\sum_{ j = 1 }^{ P }
		\left(
		\sigma_{ i j }
		\int_{ 0 }^{ T }
		\int
		g_{ k } \eta_{ B }
		\varphi
		V_{ j }
		\abs{ \nabla \chi_{ j } }
		\dd{ t } 
		-
		\frac{ 1 }{ 2 }
		\sigma_{ i j }
		\int_{ 0 }^{ T }
		\int
		g_{ k } \eta_{ B }
		\varphi^{ 2 }
		\abs{ \nabla \chi_{ j } }
		\dd{ t }
		\right)\\
		& 
		-\frac{ 1 }{ 2 }
		\int_{ 0 }^{ T }
		\left(
		\energy \left( \chi ; g_{ k } \eta_{ B } \varphi^{ 2 } \right)
		-
		\int
		g_{ k } \eta_{ B }
		\varphi^{ 2 }
		\abs{ \nabla \psi_{ i } }
		\right)
		\dd{ t }.
		\end{align*}
		Since no derivative has fallen on $ g_{ k } $, we may send $ \delta $ to 
		zero and 
		obtain by the dominated convergence theorem that we can replace $ g_{ k } $ 
		by $ \mathds{ 1 }_{ ( T_{ k - 1 } , T_{ k } ) } $. We thus end up with a 
		good summand consisting of
		\begin{equation}
		\label{the_good_summand}
		\sum_{ k = 1 }^{ K }
		\sum_{ B \in \mathcal{ B }_{ r } }
		\max_{ 1 \leq i \leq P }
		\sup_{ 
			\substack{ 
				\varphi \in \cont_{ \mathrm{c} }^{ \infty } 
				\left( ( 0 , T ) \times \flattorus \right)
				\\
				\abs{ \varphi } \leq R  
			}
		}
		\sum_{ j = 1 }^{ P }
		\sigma_{ i j }
		\int_{ T_{ k - 1 } }^{ T_{ k } } 
		\int\eta_{ B }
		\varphi 
		\left( V_{ j } - \frac{ 1 }{ 2 } \varphi \right)
		\abs{ \nabla \chi_{ j } }
		\dd{ t } 
		\end{equation}
		and an error summand given by
		\begin{align}
		\notag
		& \sum_{ k = 1 }^{ K }
		\sum_{ B \in \mathcal{ B }_{ r } }
		\max_{ 1 \leq i \leq P }
		\sup_{ \substack{ 
				\varphi \in \cont_{ \mathrm{c} }^{ \infty } 
				\left( ( 0 , T ) \times \flattorus \right)
				\\
				\abs{ \varphi } \leq R  
		} }
		- \frac{ 1 }{ 2 }
		\int_{ T_{ k - 1 } }^{ T_{ k } }
		\left(
		\energy \left( \chi ; \eta_{ B } \varphi^{ 2 } \right)
		-
		\int
		\eta_{ B }
		\varphi^{ 2 }
		\abs{ \nabla \psi_{ i } }
		\right)
		\dd{ t }
		\\
		\label{minor_error_term}
		\geq{} &
		\sum_{ k = 1 }^{ K }
		\sum_{ B \in \mathcal{ B }_{ r } }
		\max_{ 1 \leq i \leq P }
		- \frac{ R^{ 2 } }{ 2 }
		\int_{ T_{ k - 1 } }^{ T_{ k } }
		\left(
		\energy \left( \chi ; \eta_{ B } \right)
		-
		\int
		\eta_{ B }
		\abs{ \nabla \psi_{ i } }
		\right)
		\dd{ t },
		\end{align}
		where we used $ \abs{\nabla \psi_{ i }} \leq \energy ( \chi , \cdot ) $.
		We choose a majority phase $ ( i , j ) $ and estimate the 
		summands of (\ref{the_good_summand}) from below using
		\begin{align}
		\notag
		& \max_{ 1 \leq i \leq P }
		\sup_{ \substack{ 
				\varphi \in \cont_{ \mathrm{c} }^{ \infty } 
				\left( ( 0 , T ) \times \flattorus \right)
				\\
				\abs{ \varphi } \leq R  
		} }
		\sum_{ j = 1 }^{ P }
		\sigma_{ i j }
		\int_{ T_{ k  - 1 } }^{ T_{ k } }
		\int
		\eta_{ B } \varphi
		\left( V_{ j } - \frac{ 1 }{ 2 } \varphi  \right)
		\abs{ \nabla \chi_{ j } }
		\dd{ t }
		\\
		\notag
		\geq{} &
		\max_{ i < j }
		\sup_{ \substack{ 
				\varphi \in \cont_{ \mathrm{c} }^{ \infty } 
				\left( ( 0 , T ) \times \flattorus \right)
				\\
				\abs{ \varphi } \leq R  
		} }
		\sigma_{ i j }
		\int_{ T_{ k  - 1 } }^{ T_{ k } }
		\int
		\eta_{ B } \varphi
		\left( V_{ j } - \frac{ 1 }{ 2 } \varphi  \right)
		\abs{ \nabla \chi_{ j } }
		\dd{ t }
		\\
		\label{major_error}
		& \qquad \qquad \qquad \qquad -
		C \sum_{ l \notin \{ i , j \} }
		\int_{ T_{ k - 1 } }^{ T_{ k } }
		\int
		\eta_{ B } 
		\left(
		R \abs{ V_{ l } }
		+
		R^{ 2 }
		\right)
		\abs{ \nabla \chi_{ l } }
		\dd{ t }.
		\end{align}
		Concerning the error term, we have by \Cref{localization_estimate_weaker} 
		that
		\begin{equation*}
		\frac{ R^{ 2 } }{ 2 }
		\int_{ T_{ k - 1 } }^{ T_{ k } }
		\energy \left( \chi ; \eta_{ B } \right)
		-
		\int
		\eta_{ B }
		\abs{ \nabla \psi_{ i } }
		\dd{ t }
		\lesssim
		R^{ 2 }
		\sum_{ l \notin \{ i , j \} }
		\int_{ T_{ k - 1 } }^{ T_{ k } }
		\int
		\eta_{ B }
		\abs{ \nabla \chi_{ l } }
		\dd{ t },
		\end{equation*}
		which enables us to absorb the error (\ref{minor_error_term}) into the 
		error term  (\ref{major_error}).
		By applying Young's inequality, we get that for every parameter $ 
		\alpha \in ( 0 , 1 ) $, we have
		\begin{equation*}
		\int_{ T_{ k - 1 } }^{ T_{ k } }
		\int
		\eta_{ B } R \abs{ V_{ l } }
		\abs{ \nabla \chi_{ l } }
		\dd{ t }
		\lesssim
		\alpha
		\int_{ T_{ k - 1 } }^{ T_{ k } }
		\int
		\eta_{ B }
		V_{ l }^{ 2 }
		\abs{ \nabla \chi_{ l } }
		\dd{ t }
		+
		\frac{ R^{ 2 } }{ \alpha }
		\int_{ T_{ k - 1 } }^{ T_{k } }
		\int
		\eta_{ B }
		\abs{ \nabla \chi_{ l } }
		\dd{ t }.
		\end{equation*}
		Collecting our estimates, we end up with an error term which can be 
		estimated from below up to a constant by
		\begin{align*}
		- \alpha \sum_{ l = 1 }^{ P }
		\int_{ 0 }^{ T }
		\int
		V_{ l }^{ 2 }
		\abs{ \nabla \chi_{ l } }
		\dd{ t }
		-
		\frac{ R^{ 2 } }{ \alpha }
		\sum_{ k = 1 }^{ K }
		\sum_{ B \in \mathcal{ B }_{ r } }
		\max_{ i < j }
		\sum_{ l \notin \{ i , j \} }
		\int_{ T_{ k - 1 } }^{ T_{ k } }
		\int
		\eta_{ B }
		\abs{ \nabla \chi_{ l } }
		\dd{ t }.
		\end{align*}
		Moreover, we can choose for fixed $k , B $ and tuple $ (i, j ) $ a sequence 
		of test functions $ \varphi_{ n } $ with $ \abs{ \varphi_{ n } } \leq R $ 
		which converge to $ V_{ j } \mathds{ 1 }_{ \abs{ V_{ j } \leq R } } $
		in the sense that
		\begin{equation*}
		\lim_{ n \to \infty }
		\norm{ \varphi_{ n } - V_{ j } \mathds{ 1 }_{ \abs{ V_{ j } } \leq R } 
		}_{ \lp^{ 2 } \left(
			( T_{ k - 1 }, T_{ k } ) \times \flattorus ,
			\abs{ \nabla \chi_{ j } } \dd{ t }
			\right)
		}
		=
		0.
		\end{equation*}
		Combining these three arguments, we arrive at the estimate 
		\begin{align*}
		& \liminf_{ \varepsilon \to 0 } \frac{ 1 }{ 2 }
		\int_{ 0 }^{ T }
		\int
		\varepsilon \abs{ \partial_{ t } u_{ \varepsilon } }^{ 2 }
		\dd{ x }
		\dd{ t }
		\\
		\geq{} &
		- C \alpha 
		\sum_{ 1 \leq l \leq P }
		\int_{ 0 }^{ T }
		\int
		V_{ l }^{ 2 }
		\abs{ \nabla \chi_{ l } }
		\dd{ t }
		\\
		& +
		\sum_{ k = 1 }^{ K }
		\sum_{ B \in \mathcal{ B }_{ r } }
		\max_{ i < j }
		\int_{ T_{ k - 1 } }^{ T_{ k } } 
		\frac{ \sigma_{ i j } }{ 2 }
		\int
		\eta_{ B }
		\abs{ V_{ j } \mathds{ 1 }_{ \abs{ V_{ j } } \leq R 
		} }^{ 2 }
		\abs{ \nabla \chi_{ j } }
		- 
		C \frac{ R^{ 2 } }{ \alpha }
		\sum_{ l \notin \{ i , j \} }
		\int
		\eta_{ B }
		\abs{ \nabla \chi_{ l } }
		\dd{ t }
		\end{align*}
		for every parameter $ \alpha \in ( 0 , 1 ) $.
		We choose
		partitions whose width tends to zero and which are contained in each other. By the monotone convergence theorem and using Lebesgue points, we can pull the maximum inside the 
		time integral to obtain 
		\begin{align*}
		& \liminf_{ \varepsilon \to 0 } \frac{ 1 }{ 2 }
		\int_{ 0 }^{ T }
		\int
		\varepsilon \abs{ \partial_{ t } u_{ \varepsilon } }^{ 2 }
		\dd{ x }
		\dd{ t }
		\\
		\geq{}  & 
		- \alpha C
		\sum_{ 1 \leq l \leq P }
		\int_{ 0 }^{ T }
		\int
		V_{ l }^{ 2 }
		\abs{ \nabla \chi_{ l } }
		\dd{ t }
		\\
		& +
		\int_{ 0 }^{ T }
		\sum_{ B \in \mathcal{ B }_{ r } }
		\max_{ i < j }
		\frac{ \sigma_{ i j } }{ 2 }
		\int
		\eta_{ B }
		\abs{ V_{ j } \mathds{ 1 }_{ \abs{ V_{ j } } \leq R 
		} }^{ 2 }
		\abs{ \nabla \chi_{ j } }
		- 
		C \frac{ R^{ 2 } }{ \alpha }
		\sum_{ l \notin \{ i , j \} }
		\int
		\eta_{ B }
		\abs{ \nabla \chi_{ l } }
		\dd{ t }
		\\
		\geq{} &
		\frac{ 1 }{ 2 }
		\sum_{ 1 \leq i < j \leq P }
		\sigma_{ i j }
		\int_{ 0 }^{ T }
		\int_{ \Sigma_{ i j } }
		V_{ i }^{ 2 } \mathds{ 1 }_{ \abs{ V_{ i } } \leq R }
		\dd{ \hm^{ d - 1 } }
		\dd{ t }
		\\
		& -
		C \left(
		\alpha \sum_{ 1 \leq l \leq P }
		\int_{ 0 }^{ T }
		\int
		V_{ l }^{ 2 }
		\abs{ \nabla \chi_{ l } }
		\dd{ t }
		+
		\frac{ R^{ 2 } }{ \alpha }
		\int_{ 0 }^{ T }
		\sum_{ B \in \mathcal{ B }_{ r } }
		\min_{ i \neq j }
		\sum_{ k \notin \{ i , j \} }
		\int
		\eta_{ B }
		\abs{ \nabla \chi_{ k } }
		\dd{ t }
		\right).
		\end{align*}
		We first send $ r \to 0 $ and then 
		$ \alpha \to 0$ . By the localization result \Cref{localization_lemma_with_normals} we thus obtain that for all $R > 0 $
		\begin{equation*}
		\liminf_{ \varepsilon \to 0 } \frac{ 1 }{ 2 }
		\int_{ 0 }^{ T }
		\int
		\varepsilon \abs{ \partial_{ t } u_{ \varepsilon } }^{ 2 }
		\dd{ x }
		\dd{ t }
		\geq
		\frac{ 1 }{ 2 }
		\sum_{ 1 \leq i < j \leq P }
		\sigma_{ i j }
		\int_{ 0 }^{ T }
		\int_{ \Sigma_{ i j } }
		V_{ i }^{ 2 }
		\mathds{ 1 }_{ \abs{ V_{ i } } \leq R }
		\dd{ \hm^{ d - 1 } }
		\dd{ t }.
		\end{equation*}
		Therefore the lower semicontinuity of the velocity term now follows from 
		the 
		monotone convergence theorem.
	\end{proof}
	
	\subsection{De Giorgi type varifold solutions for mean curvature flow}
	\label{subsection_de_giorgi_type_varifold_solutions_for_mcf}
	
	Up until this point, we have always made the crucial assumption of energy 
	convergence (\ref{energy_convergence}) for our proofs.
	However this is usually a very strong assumption. One thing which could go 
	wrong is for example illustrated in 
	\Cref{figure_interfaces_collapse}. There we see that two approximate interfaces 
	of 
	the first and second phase collapse as $ \varepsilon $ tends to zero. This 
	results in a loss of energy since the measure theoretic boundary of a set with
	finite perimeter does not see such lines.
	
	\begin{figure}[h]
		\centering
		
		\begin{subfigure}[b]{0.3\linewidth}
			\begin{tikzpicture}
			\filldraw[fill=yellow, draw=yellow] (0,0) rectangle (0.75,3);
			\shade[left color=yellow,right color=red] (0.75,0) rectangle (1.125,3);
			\filldraw[fill=red, draw=red] (1.125,0) rectangle (1.875,3);
			\shade[left color=red,right color=yellow] (1.875,0) rectangle (2.25,3);
			\filldraw[fill=yellow, draw=yellow] (2.25,0) rectangle (3,3);
			
			\node at (0.375,1.5) { $\approx \alpha_{ 1 } $};
			\node at (1.5,1.5) {\large $\approx \alpha_{ 2 } $};
			\node at (2.625,1.5) { $\approx \alpha_{ 1 } $};
			
			\end{tikzpicture}
			
			\caption{$ 1 > \varepsilon > 0 $}
			\label{subfigure_epsilon}
		\end{subfigure}
		\hfill
		\begin{subfigure}[b]{0.3\linewidth}
			\begin{tikzpicture}
			\filldraw[fill=yellow, draw=yellow] (0,0) rectangle (1.125,3);
			\shade[left color=yellow,right color=red] (1.125,0) rectangle (1.3125,3);
			\filldraw[fill=red, draw=red] (1.3125,0) rectangle (1.6875,3);
			\shade[left color=red,right color=yellow] (1.6875,0) rectangle (1.875,3);
			\filldraw[fill=yellow, draw=yellow] (1.875,0) rectangle (3,3);
			
			\node at (0.5625,1.5) {\Large $\approx \alpha_{ 1 } $};
			\node at (1.5,1.5) {\small $\approx \alpha_{ 2 } $};
			\node at (2.4375,1.5) {\Large $\approx \alpha_{ 1 } $};
			
			\end{tikzpicture}
			
			\caption{$ 1 \gg  \varepsilon >0 $}
			\label{subfigure_varepsilon_prime}
		\end{subfigure}
		\hfill
		\begin{subfigure}[b]{0.3\linewidth}
			\begin{tikzpicture}
			\filldraw[fill=yellow, draw=yellow] (0,0) rectangle (3,3);
			\draw[red,thick](1.5,0)--(1.5,3);
			
			\node at (0.75,1.5) {\LARGE $= \alpha_{ 1 } $};
			\node at (2.25,1.5) {\LARGE $= \alpha_{ 1 } $};
			
			\end{tikzpicture}
			
			\caption{$ \varepsilon = 0 $}
			\label{subfigure_varepsilon_zero}
		\end{subfigure}
		
		\caption{Profile of the solution $u_{ \varepsilon } $ as $ \varepsilon $ 
			tends to zero.}
		\label{figure_interfaces_collapse}
	\end{figure}
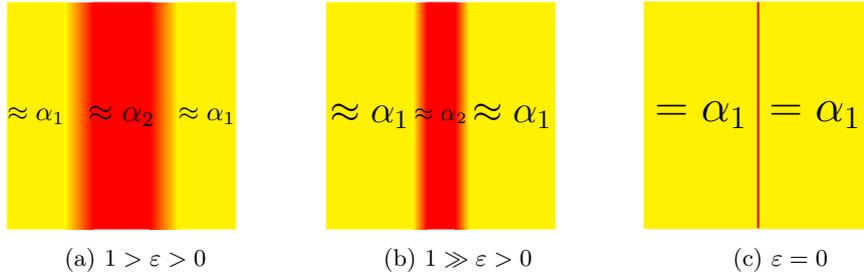
	
	We now introduce the solution concept by Hensel and Laux in 
	\cite{hensel_laux_varifold_solution_concept_for_mean_curvature_flow} which 
	tackles this issue. 
	For the two-phase case, the definition is as follows.
	
	\begin{definition}[De Giorgi type varifold solution for two-phase mean 
		curvature flow]
		\label{de_giorgi_varifold_solution_for_mcf}
		Let $ T < \infty $ be an arbitrary finite time horizon and let $ \mu = 
		\lm^{ 1 } \otimes ( \mu_{ t } )_{ t \in ( 0 , T ) } $ be a 
		family of oriented varifolds $ \mu_{ t } \in \mathcal{ M } \left( 
		\flattorus \times \mathbb{ S }^{ d - 1 } \right) $ for $ t \in ( 0 , T ) 
		$ such that $ t \mapsto \int_{ \flattorus \times \mathbb{ S }^{ d 
				-1 } } \eta (t, x , p ) \dd \mu_{ t } ( x , p ) $ is measurable 
		for all 
		$ \eta \in \lp^{ 1 } \left( ( 0 , T ) ; \cont \left( \flattorus, 
		\mathbb{ S 
		}^{ d - 1 } \right) \right) $. 
		Consider also a family $ A = ( A_{ t } )_{ t \in ( 0 , T ) } $ of 
		subsets of $ \flattorus $ with finite perimeter such that the associated 
		indicator function $ \chi ( x , t ) \coloneqq \chi_{ A _{ t } } ( x ) $ 
		is in the space
		$ \lp^{ \infty } \left( ( 0 , T ) ; \bv ( \flattorus ; \{ 0 , 
		1 \} ) \right) $.
		Let $ \sigma > 0 $ be a surface tension constant.
		
		Given an initial energy
		$ \omega^{ 0 } \in \mathcal{ M } \left( \flattorus \right) $ and initial 
		data $ \chi^{ 0 } \in \bv 
		\left( \flattorus ; \{ 0 , 1 \} \right)$, we call the pair $ ( \mu , \chi ) 
		$ a \emph{De Giorgi type varifold solution to two-phase mean curvature 
			flow 
			with initial data} $ ( \omega^{ 0 } , \chi^{ 0 } ) $ if the following 
		holds.
		\begin{enumerate}
			\item (Existence of a normal speed)
			Writing $ \mu_{ t } = \omega_{ t } \otimes ( \lambda_{ t, x } )_{ x 
				\in \flattorus} $ for the disintegration of $ \mu_{ t } $, we 
			require 
			the existence of some 
			$ V \in \lp^{ 2 } \left( ( 0 , T ) \times \flattorus ,
			\omega_{ t } \right) $ encoding a normal velocity in the sense of
			\begin{align}
			\notag
			&\sigma
			\int
			\chi ( T' , x ) \varphi ( T' , x ) 
			-
			\chi^{ 0 } ( x ) \varphi (0,x)
			\dd{ x }
			\\
			\label{equation_varifold_velocity}
			={} &
			\sigma
			\int_{ 0 }^{ T' }
			\int
			\chi
			\partial_{ t } \varphi 
			\dd{ x }
			\dd{ t }
			+
			\int_{ 0 }^{ T' }
			\int
			V \varphi 
			\dd{ \omega_{ t } }
			\dd{ t }
			\end{align}
			for almost every $ T' \in ( 0 , T ) $ and all $ \varphi \in \cont_{ 
				\mathrm{c} }^{ \infty } \left( [ 0 , T ) \times \flattorus \right) 
			$.
			
			\item (Existence of a generalized mean curvature vector)
			We require the existence of some 
			$ H \in \lp^{ 2 } \left( 
			( 0 , T ) \times \flattorus , \omega_{ t } ; \mathbb{ R }^{ d } \right) 
			$
			encoding a generalized mean curvature vector by
			\begin{equation}
			\label{equation_varifold_mean_curvature}
			\int_{ 0 }^{ T }
			\int
			\inner*{ H }{ \xi }
			\dd{ \omega_{ t } }
			\dd{ t }
			=
			-
			\int_{ 0 }^{ T }
			\int_{ \flattorus \times \mathbb{ S }^{ d - 1 } }
			\inner*{ \xi }{ \mathrm{Id} - p \otimes p }
			\dd{ \mu_{ t } ( x, p ) }
			\dd{ t }
			\end{equation}
			for all $ \xi \in \cont_{ \mathrm{c} }^{ \infty } \left( [ 0, T ) 
			\times \flattorus ; \mathbb{ R }^{ d } \right) $.
			
			\item (De Giorgi type optimal energy dissipation inequality)
			A sharp energy dissipation inequality holds in form of
			\begin{equation}
			\label{equation_varifold_de_giorgi_inequality}
			\omega_{ T' } ( \flattorus )
			+
			\frac{ 1 }{ 2 }
			\int_{ 0 }^{ T' }
			\int
			V^{ 2 }
			+
			\abs{ H }^{ 2 }
			\dd{ \omega_{ t } }
			\dd{ t }
			\leq
			\omega^{ 0 } ( \flattorus )
			\end{equation}
			for almost every $ T' \in ( 0 , T ) $.
			
			\item (Compatibility)
			For almost every $ t \in ( 0 , T ) $ and all $ \xi \in \cont^{ 
				\infty } \left( \flattorus ; \mathbb{ R }^{ d } \right) $, it holds 
			that
			\begin{equation}
			\label{equation_varifold_compatibility}
			\sigma
			\int
			\inner*{ \xi }
			{ \nabla \chi ( t , \cdot ) }
			=
			\int_{ \flattorus \times \mathbb{ S }^{ d - 1 } }
			\inner*{ \xi }{ p }
			\dd{ \mu_{ t } ( x, p ) }.
			\end{equation}
		\end{enumerate}
	\end{definition}
	
	We firstly want to discuss this definition. If we have a De Giorgi type $ \bv 
	$-solution $ \chi $ to two-phase mean curvature flow in the sense of 
	\Cref{de_giorgi_solution_to_mmcf}, we can think of 
	the 
	oriented varifold $ \mu $ as the measure 
	\begin{equation} 
	\label{equation_simple_varifold}
	\mu =
	\sigma \lm^{ 1 } |_{ 
		( 0 , T ) } \otimes ( \abs{ \nabla \chi  ( t ) } )_{ t \in ( 0 , T ) }
	\otimes  ( \delta_{ \nu ( t , x ) } )_{ t \in ( 0 , T ), x \in \flattorus 
	}
	\end{equation}
	and therefore the measure $ \omega_{ t } $ 
	becomes the energy measure $ \energy ( \chi ( t ) , \cdot ) $.
	The advantage of the varifold formulation is that our new energy measure $ 
	\omega_{ t } 
	$ is not restricted to only seeing the measure theoretic boundary of $ \chi $, 
	but can actually capture phenomena as described in 
	\Cref{figure_interfaces_collapse}.
	For example in such a scenario we would expect the measure $ \mu_{ t } $ to be 
	defined by
	\begin{equation}
	\label{equation_mu_in_figure}
	\mu_{ t } \coloneqq
	2 \sigma \hm^{ 1 } |_{ l } \otimes \left(  \frac{ 1 }{ 2 } \delta_{ e_{ 1 } 
	} + 
	\frac{ 1 }{ 2 }\delta_{ -e_{ 1 } }\right)_{ x \in \flattorus } ,
	\end{equation}
	where $ l $ is the red line to which the phase 
	of $ \alpha_{ 2 } $ shrank down as $ \varepsilon $ approached zero and $ e_{ 1 
	} = ( 1 , 0 )^{ \top } $. The factor 
	2 comes from the fact that we obtain energy from both of the collapsing 
	interfaces.
	
	The equation (\ref{equation_varifold_velocity}) for the  normal speed is simply 
	motivated through 
	the fundamental theorem of calculus. Assuming that everything is nice and 
	smooth, we can compute that 
	\begin{align*}
	&\sigma \int
	\chi ( T , x ) \varphi ( T , x ) - \chi^{ 0 } ( x ) \varphi ( 0 , x )
	\dd{ x }
	\\
	={} &
	\sigma \int_{ 0 }^{ T }
	\int
	\partial_{ t } \left(
	\chi ( t , x ) \varphi ( t , x )
	\right)
	\dd{ x }
	\dd{ t }
	\\
	={} &
	\sigma \int_{ 0 }^{ T }
	\int
	\partial_{ t } \chi ( t , x ) \varphi ( t , x )
	+
	\chi ( t , x ) \partial_{ t } \varphi ( t , x )
	\dd{ x }
	\dd{ t }
	\\
	={} &
	\int_{ 0 }^{ T }
	\int
	V \varphi 
	\dd{ \omega_{ t} }
	\dd{ t }
	+
	\sigma \int_{ 0 }^{ T }
	\int
	\chi ( t , x )
	\partial_{ t } \varphi ( t , x )
	\dd{ x }
	\dd{ t },
	\end{align*}
	where for the last equality, we used $ \partial_{ t } \chi = V \abs{ \nabla 
		\chi } \dd{ t } $ and $ \omega_{ t } = \sigma \abs{ \nabla \chi ( t ) } $.
	
	The equation (\ref{equation_varifold_mean_curvature}) for the generalized mean 
	curvature is straightforward. In fact if we
	assume that the varifold is given through equation 
	(\ref{equation_simple_varifold}), then we have that the right hand side of 
	equation 
	(\ref{equation_varifold_mean_curvature}) reads for a fixed time $ t $
	\begin{equation*}
	\int_{ \flattorus \times \mathbb{ S }^{ d - 1 } }
	\inner*{ \xi }{ \mathrm{Id} - p \otimes p }
	\dd{ \omega_{ t } }
	=
	\sigma
	\int
	\inner*{ \xi }{\mathrm{Id} - \nu \otimes \nu }
	\abs{ \nabla \chi }.
	\end{equation*}
	This is exactly the distributional formulation for the mean curvature vector.
	
	De Giorgis inequality (\ref{equation_varifold_de_giorgi_inequality}) is 
	self-explanatory since $ \omega_{ t } $ is the energy measure. The 
	compatibility condition (\ref{equation_varifold_compatibility}) is necessary to 
	couple the evolving set $ A_{ t } $ to 
	the varifold. Notice that in our above example (\ref{equation_mu_in_figure}), 
	this condition is still satisfied.
	Even though the energy measure $ 
	\omega_{ t } $ sees the red strip 
	and the measure theoretic boundary of the corresponding indicator function does 
	not, the term on the right hand side of 
	(\ref{equation_varifold_compatibility}) 
	is zero since
	\begin{equation*}
	\int_{ \flattorus \times \mathbb{ S }^{ d - 1 } }
	\inner*{ \xi }{ p }
	\dd{ \mu_{ t } }
	=
	\sigma
	\int_{ l }
	\inner*{ \xi }{ e_{ 1 } - e_{ 1 } }
	\dd{ \hm^{ 1 } }
	= 
	0.
	\end{equation*}
	
	For the multiphase case, we propose the following solution concept, which 
	generalizes
	\cite[Def.~2]{hensel_laux_varifold_solution_concept_for_mean_curvature_flow} to 
	the case of arbitrary surface tensions.
	
	\begin{definition}[De Giorgi type varifold solution to multiphase mean 
		curvature flow]
		\label{de_giorgi_varifold_solutions_for_mmcf}
		Let $T < \infty $ be an arbitrary finite time horizon and let $ P \in 
		\mathbb{ N }_{ \geq 2 } $ be the number of phases. For each 
		pair of phases $ ( i , j ) \in \{ 1 , \dotsc, P \}^{ 2 } $, let 
		$ \mu_{ i j } = \lm^{ 1 }|_{ ( 0 , T ) } \otimes ( \mu_{ t , i j } )_{ t 
			\in ( 0 , T ) } $ be a family of oriented 
		varifolds 
		$ \mu_{ t , ij } \in \mathcal{ M } \left(
		\flattorus \times \mathbb{ S }^{ d - 1 }
		\right) $ for
		$ t \in ( 0 , T ) $ such that the map 
		$ t \mapsto \int_{ \flattorus \times \mathbb{ S }^{ d - 1 } }
		\eta ( t , x , p )
		\dd{ \mu_{ t ,i j } ( x, p ) }$
		is measurable for all
		$ \eta \in \lp^{ 1 } \left(	
		( 0 , T ) ; \cont \left( \flattorus \times \mathbb{ S }^{ d- 1 } 
		\right) 
		\right) $.
		Define the evolving oriented varifolds $ \mu_{ i } = \lm^{ 1 }|_{ ( 0 , T 
			) } \otimes ( \mu_{ t , i } )_{ t \in ( 0 , T ) } $ for $ i \in \{ 1, 
		\dotsc, P \} $ and $ \mu = \lm^{ 1 }|_{ ( 0 , T ) } \otimes ( \mu_{ t 
		} )_{ t \in ( 0 , T ) } $ by 
		\begin{equation}
		\label{moon_equation}
		\mu_{ t , i }
		\coloneqq
		2 \mu_{ t , i i }
		+
		\sum_{ j = 1 , j \neq i }^{ P }
		\mu_{ t , i j }
		\quad \text{and} \quad
		\mu_{ t }
		\coloneqq
		\frac{ 1 }{ 2 }
		\sum_{ i = 1 }^{ P }
		\mu_{ t, i }.
		\end{equation}
		The disintegration of $ \mu_{ t , i j } $ is expressed in form of 
		$ \mu_{ t , i j } = \omega_{ t ,i j } \otimes \left( \lambda _{ t , x , i j 
		} \right)_{ x \in \flattorus } $ with expected value
		$ \langle \lambda_{ t , x , i j } \rangle 
		\coloneqq
		\int_{ \mathbb{ S }^{ d - 1 } }
		p 
		\dd{ \lambda_{ t , x , i j } ( p ) } $.
		Analogous expressions are introduced for the disintegrations of $ \mu_{ t, 
			i } $ and $ \mu_{ t } $.
		
		Furthermore consider a tuple 
		$ A = \left( A_{ 1 } , \dotsc , A_{ P } \right) $ 
		such that for each phase $ 1 \leq i \leq P $, we have a family
		$ A_{ i } = ( A_{ i } ( t ) )_{ t \in ( 0 , T ) } $ of subsets of $ 
		\flattorus $ with finite perimeter. We also require 
		$ \left( A_{ 1 } ( t ) , \dotsc, A_{ P } ( t ) \right) $
		to be a partition of $ \flattorus $ for all $ t \in ( 0 , T ) $ and 
		that for each $ 1 \leq i \leq P $, the associated indicator function 
		satisfies 
		$ \chi_{ i } \in \lp^{ \infty } \left(
		( 0 , T ) ;
		\bv \left( \flattorus ; \{ 0 , 1 \} \right)
		\right) $.
		We shortly write $ \chi = \left( \chi_{ 1 } , \dotsc, \chi_{ P } \right) $.
		
		Given initial data 
		$ (
		\omega^{ 0 },
		( \chi_{ i}^{ 0 } )_{ 1 \leq i \leq P}
		) $
		of the above form and a $ (P \times P) $-matrix of surface tensions $ 
		\sigma 
		$ such that $ \sigma_{ i j } > 0 $ for $ i \neq j 
		$, we call the pair $ \left( \mu , \chi \right) $ a
		\emph{De Giorgi type varifold solution to multiphase mean curvature flow 
			with initial data} $ ( \omega^{ 0 } , \chi^{ 0 } ) $ \emph{and surface 
			tensions} $ \sigma $ if the following requirements hold true.
		\begin{enumerate}
			\item (Existence of normal speeds)
			For each phase $ 1 \leq i \leq P $, there exists a normal speed
			$ V_{ i } \in \lp^{ 2 } \left(
			( 0 , T ) \times \flattorus, \omega_{ i } \right) $ in the sense 
			that
			\begin{align}
			\notag
			&
			\int
			\chi_{ i } ( T', x ) \varphi ( T', x ) 
			-
			\chi_{ i }^{ 0 } ( x ) \varphi ( 0 , x )
			\dd{ x }
			\\
			\label{equation_varifold_velocity_multiphase}
			={} &
			\int_{ 0 }^{ T' }
			\int
			\chi_{ i }
			\partial_{ t } \varphi
			\dd{ x }
			\dd{ t }
			+
			\sum_{ j =1, j \neq i  }^{ P }
			\frac{ 1 }{ \sigma_{ i j } }
			\int_{ 0 }^{ T' }
			\int
			V_{ i }
			\varphi
			\dd{ \omega_{ t , i j } }
			\dd{ t }
			\end{align}
			for almost every $ T' \in ( 0 , T ) $ and all $ \varphi \in \cont_{ 
				\mathrm{c} }^{ \infty } \left( [ 0 , T ) \times \flattorus\right) 
			$.
			
			\item (Existence of a generalized mean curvature vector)
			There exists a generalized mean curvature vector 
			$ H \in \lp^{ 2 } \left(
			( 0 , T )\times \flattorus , \omega; \mathbb{ R }^{ d } \right)
			$
			in the sense that
			\begin{equation}
			\label{equation_varifold_mean_curvature_multiphase}
			\int_{ 0 }^{ T }
			\int
			\inner*{ H }{ \xi }
			\dd{ \omega_{ t } }
			\dd{ t }
			=
			-
			\int_{ 0 }^{ T }
			\int_{ \flattorus \times \mathbb{ S }^{ d - 1 } }
			\inner*{ \diff \xi }
			{ \mathrm{Id} - p \otimes p }
			\dd{ \mu_{ t } ( x , p ) }
			\dd{ t }
			\end{equation}
			holds for all $ \xi \in \cont_{ 
				\mathrm{c} }^{ \infty } \left( [ 0 , T ) \times \flattorus ; 
			\mathbb{ 
				R }^{ d } \right) $.
			
			\item (De Giorgi type optimal energy dissipation inequality)
			A sharp energy dissipation inequality holds in form of
			\begin{equation}
			\label{equation_varifold_energy_dissipation_inequality}
			\omega_{ T' } ( \flattorus )
			+
			\frac{ 1 }{ 2 }
			\sum_{ i = 1 }^{ P }
			\int_{ 0 }^{ T' }
			\int
			V_{ i }^{ 2 }
			\frac{ 1 }{ 2 }
			\dd{ \omega_{ t , i } }
			\dd{ t }
			+
			\frac{ 1 }{ 2 }
			\int_{ 0 }^{ T' }
			\int
			\abs{ H }^{ 2 }
			\dd{ \omega_{ t } }
			\dd{ t }
			\leq
			\omega^{ 0 } ( \flattorus )
			\end{equation}
			for almost every $ T' \in ( 0 , T ) $.
			
			\item (Compatibility conditions)
			For all $ 1 \leq i , j \leq P $, we require 
			\begin{equation}
			\label{varifold_symmetry_of_energy_measures}
			\omega_{ t , i j }
			=
			\omega_{ t , j i } 
			\end{equation}
			for almost every $ t \in ( 0 , T ) $ and
			\begin{align}
			\label{varifold_symmetry_expectations}
			\langle \lambda_{ t, x , i j } \rangle
			&= 
			- \langle \lambda_{ t, x  ,j i } \rangle,
			\\
			\label{varifold_symmetry_velocities}
			V_{ i } ( t, x ) 
			&= 
			- V_{ j } ( t, x ),
			\\
			\label{varifold_mean_curvature_points_in_normal_direction}
			\abs{ \langle \lambda_{ t , x , i j } \rangle }^{ 2 }
			H ( t, x )
			& =
			\inner*{ H( t , x ) }{ \langle \lambda_{ t , x , i j } \rangle }
			\langle \lambda_{ t , x }^{ i j } \rangle
			\end{align}
			for almost every $ t \in ( 0 , T ) $ and $ \omega_{ t , i j } $ 
			almost every $ x \in \flattorus $. Finally for all $ 1 \leq i \leq P $, 
			we 
			have
			\begin{equation}
			\label{varifold_compatibility_condition_multiphase}
			\int
			\inner*{ \xi }{ \nabla \chi_{ i } ( t , \cdot ) }
			=
			\sum_{ j = 1 , j \neq i }^{ P }
			\frac{ 1 }{ \sigma_{ i j } }
			\int_{ \flattorus \times \mathbb{ S }^{ d - 1 } }
			\inner*{ \xi }{ p }
			\dd{ \mu_{ t , i j } }
			\end{equation}
			for almost every $ t \in ( 0 , T ) $ and every $ \xi \in \cont^{ 
				\infty } \left( \flattorus ; \mathbb{ R }^{ d } \right) $.
		\end{enumerate}
	\end{definition}
	
	As before, we first want to motivate this definition. If we have a De Giorgi 
	type $ \bv $-solution to multiphase mean curvature flow in the sense of 
	\Cref{de_giorgi_solution_to_mmcf}, we 
	can think of the varifold $ \mu_{ t , i j } $ for $ i \neq j $ as
	\begin{equation}
	\label{varifold_simplest_case_mutliphase}
	\mu_{ t , i j }
	=
	\sigma_{ i j }
	\hm^{ d - 1 }\llcorner_{ \Sigma_{ i j } } 
	\otimes
	( \delta_{ \nu_{ i } } )_{ x \in \flattorus }
	\end{equation}
	and $ \mu_{ t , i i } = 0 $.
	But if the approximate interfaces collapse as in 
	\Cref{figure_interfaces_collapse}, then this will be captured by the 
	varifolds $ \mu_{ t, i i }$, which will be 
	described by
	\begin{equation}
	\label{varifold_when_interfaces_collapse}
	\mu_{ t, 1 1 }
	=
	2
	\hm^{ 1 } \llcorner_{ l }
	\otimes
	\left( 
	\frac{ 1 }{ 2 } \delta_{ e_{ 1 } } 
	+ 
	\frac{ 1 }{ 2 } \delta_{ - e_{ 1 } } 
	\right)_{ x \in \flattorus }
	\end{equation}
	as in the two-phase case.
	
	Considering the definition of $ \mu_{ t } $ by equality (\ref{moon_equation}), 
	the factor $1/2$ is present since every 
	interface $ \mu_{ t , i j } $, $ i \neq j $, is counted twice. Thus we need the 
	factor 2 in front of $ \mu_{ t , i i } $ in the definition of $ \mu_{ t , i } $
	since this interface is not counted twice.
	
	The equation (\ref{equation_varifold_velocity_multiphase}) is similar to the 
	two-phase case. Note however that we only consider the energy measures $ 
	\omega_{ t , i j } $ for $ i \neq j $ since we do not expect $ \omega_{ t , i i 
	} $ to be relevant for the motion of the $ i $-th phase. Equation
	(\ref{equation_varifold_mean_curvature_multiphase}) is as in the two-phase 
	case. For the energy dissipation inequality, we note again that we need the factor $ 
	1/2 $ in front of $ \omega_{ t , i } $ since each interface will be counted 
	twice. 
	
	The first compatibility condition (\ref{varifold_symmetry_of_energy_measures}) 
	follows simply from $ \sigma_{ i j } = \sigma_{ j i }$ and $ \Sigma_{ i j } = 
	\Sigma_{ j i } $ if we have a De Giorgi type $ \bv $-solution. But it should 
	even 
	hold true in situations like 
	\Cref{figure_interfaces_collapse}, since the equation 
	(\ref{varifold_symmetry_of_energy_measures}) becomes trivial for $ i = 
	j $. 
	In the same situation, the second compatibility condition 
	(\ref{varifold_symmetry_expectations}) states that even if approximate 
	interfaces collapse, we have $ \nu_{ i } = - \nu_{ j } $ on $ \Sigma_{ i j } $ 
	$ \hm^{ 
		d -1 } $-almost everywhere. 
	Indeed, for $ i = j $, the equation states that 
	$ \langle \lambda_{ t , x , i i } \rangle = 0 $, which is satisfied by equation
	(\ref{varifold_when_interfaces_collapse}). 
	Similarly we explain the anti-symmetry of the velocities 
	(\ref{varifold_symmetry_velocities}). The compatibility condition 
	(\ref{varifold_mean_curvature_points_in_normal_direction}) encodes that the 
	mean curvature vector should always point in direction of the inner unit 
	normal, which is not necessarily the case for varifolds, as demonstrated in 
	\Cref{mean_curvature_vector_does_not_have_to_point_in_normal_direction}.
	Nevertheless this can always be assured if we have a De Giorgi type $ \bv 
	$-solution, see the proof of
	\Cref{bv_solutions_are_varifold_solutions}. 
	For the last compatibility condition 
	(\ref{varifold_compatibility_condition_multiphase}), we again notice that it 
	behaves similarly to its two-phase equivalent.
	
	\begin{theorem}
		\label{bv_solutions_are_varifold_solutions}
		Every  De Giorgi type $ \bv $-solution to multiphase mean curvature 
		flow in the sense of \Cref{de_giorgi_solution_to_mmcf} is also a De Giorgi 
		type varifold solutions for multiphase mean curvature flow in the sense of
		\Cref{de_giorgi_varifold_solutions_for_mmcf}.
		The varifold $ \mu $ is given by
		\begin{align*}
		\mu_{ t, ij  }
		& \coloneqq
		\sigma_{ i j }
		\hm^{ d- 1 }\llcorner_{ \Sigma_{ i j } ( t ) }
		\otimes
		( \delta_{ \nu_{ i } ( t, x ) } )_{ x \in \flattorus }
		\end{align*} 
		for $ i \neq j $, $
		\mu_{ t , i i } = 0 $ and the initial energy is $ \omega^{ 0 } 
		= \energy ( \chi^{ 0 } , \cdot ) $.
	\end{theorem}
	
	\begin{proof}
		We first note that
		\begin{align*}
		\mu_{ t , i }
		& =
		\sum_{ j \neq i }
		\sigma_{ i j }
		\hm^{ d- 1 }\llcorner_{ \Sigma_{ i j } }
		\otimes
		( \delta_{ \nu_{ i } } )_{ x \in \flattorus }
		\shortintertext{and}
		\mu_{ t }
		& =
		\sum_{ 1 \leq i < j \leq P }
		\sigma_{ i j }
		\hm^{ d - 1 } \llcorner_{ \Sigma_{i j } }
		\otimes
		\left( \frac{ 1 }{ 2 } \delta_{ \nu_{ i } } + \frac{ 1 }{ 2 } \delta_{ 
			\nu_{ j } } \right)_{ x \in \flattorus }.
		\end{align*}
		Thus the energy measures are given by
		\begin{align*}
		\omega_{ t , i } 
		&= 
		\sum_{ j \neq i }
		\sigma_{ i j }
		\hm^{ d - 1 } \llcorner_{ \Sigma_{ i j } ( t ) }
		\shortintertext{and}
		\omega_{ t }
		& =
		\sum_{ 1 \leq i < j \leq P }
		\sigma_{ i j }
		\hm^{ d - 1 } \llcorner_{ \Sigma_{ i j } ( t ) }
		=
		\energy ( \chi ( t ) ; \cdot ).
		\end{align*}
		The requirement 
		$ \chi_{ i } \in \lp^{ \infty } \left(
		( 0 , T ) ; \bv ( \flattorus ; \{ 0 , 1 \} ) \right) $
		is an immediate consequence of the energy dissipation inequality 
		(\ref{optimal_energy_dissipation_solution}).
		The existence of normal velocities is given, but we need to check that
		equation (\ref{equation_varifold_velocity_multiphase}) holds.
		If we assume that $ \varphi $ is compactly supported in $ ( 0 , T ) \times 
		\flattorus $, then the equation 
		follows by approximating $ \chi_{ i } $ with $ \rho_{ n } \ast \chi_{ i } 
		$, where $ \rho_{ n } $ is a sequence of radial symmetric standard 
		mollifiers. 
		In general, we have the problem that $ \partial_{ t } \chi_{ i } = V_{ i } 
		\abs{ \nabla \chi_{ i } } \dd{ t } $ holds only tested against functions 
		supported in $ ( 0 , T ) \times \flattorus $.
		Therefore we take a sequence $ \eta_{ n } $ of non-decreasing smooth 
		functions with compact 
		support in $ ( 0 , T ] $ which are equal to $ 1 $ on $ ( 1/ n , T ] $. Then 
		we have
		\begin{align*}
		&\int
		\chi_{ i } ( T' , x ) \varphi ( T', x )
		\dd{ x }\\
		={} &
		\lim_{ n \to \infty }
		\int
		\chi_{ i } ( T', x ) \eta_{ n } ( T' ) \varphi ( T' , x )
		\dd{ x }
		\\
		={} &
		\lim_{ n \to \infty }
		\int_{ 0 }^{ T' }
		\int
		\eta_{ n } \varphi V_{ i }
		\abs{ \nabla \chi^{ i } }
		\dd{ t }
		+
		\int_{ 0 }^{ T' }
		\int
		\chi_{ i }
		\eta_{ n }
		\partial_{ t } \varphi 
		\dd{ x }
		\dd{ t }
		+
		\int_{ 0 }^{ T' }
		\int
		\chi_{ i }
		\partial_{ t } \eta_{ n }
		\varphi
		\dd{ x }
		\dd{ t }
		\end{align*}
		for almost every $ 0 < T' < T $.
		Using the dominated convergence theorem, we see that the first two summands 
		converge to the right hand side of the velocity equation 
		(\ref{equation_varifold_velocity_multiphase}). For the third summand, we 
		compute that since $ \int_{ 0 }^{ T } \eta' \dd{ t } = 1 $, we have
		\begin{align*}
		& \abs{
			\int_{ 0 }^{ T' }
			\partial_{ t } \eta_{ n }
			\int
			\chi_{ i }
			\varphi
			\dd{ x }
			\dd{t}
			-
			\int
			\chi^{ 0 }_{ i } ( x )
			\varphi ( 0 , x )
			\dd{ x }
		}
		\\
		={} &
		\abs{ 
			\int_{ 0 }^{ T' }
			\partial_{ t } \eta_{ n }
			\int
			\chi_{ i }
			\varphi
			-
			\chi^{ 0 }_{ i }
			\varphi ( 0 , x )
			\dd{ x }
			\dd{ t }
		}
		\\
		\lesssim{} &
		\int_{ 0 }^{ T' } 
		\partial_{ t } \eta_{ n }
		\left(
		\int
		\abs{ \chi^{ i } \varphi - \chi^{ 0 }_{ i } \varphi ( 0 , x 
			) 
		}^{ 2 }
		\dd{ x }
		\right)^{ 1/2 }
		\dd{ t }
		\\
		\leq {} &
		\sup_{ t \in ( 0, \frac{ 1 }{ n } ) }
		\left(
		\int
		\abs{ \chi_{ i } \varphi - \chi_{ i }^{ 0 } \varphi ( 0 , x 
			) }^{ 2 }
		\dd{ x }
		\right)^{ 1/2 }.
		\end{align*}
		This converges to zero since $ \chi_{ i } $ attains the initial data 
		continuously with respect to the $ \lp^{ 2 } $-norm. Therefore the velocity 
		equation 
		(\ref{equation_varifold_velocity_multiphase}) follows.
		
		The curvature equation (\ref{equation_varifold_mean_curvature_multiphase}) 
		follows 
		immediately from equation (\ref{mean_curvature_vector_bv_de_giorgi})
		since
		\begin{equation*}
		\int_{ \flattorus \times \mathbb{ S }^{ d - 1 } }
		\inner*{ \diff \xi }{ \mathrm{Id} - p \otimes p }
		\dd{ \mu_{ t } }
		=
		\sum_{ 1 \leq i < j \leq P }
		\sigma_{ i j }
		\int_{ \Sigma_{ i j } }
		\inner*{ \diff \xi }{\mathrm{Id} - \nu_{i } \otimes \nu_{ i } }
		\dd{ \hm^{ d - 1 } }.
		\end{equation*}
		We used here that on $ \Sigma_{ i j } $, 
		we have $ \nu_{i } \otimes \nu_{ i } = \nu_{ j } \otimes \nu_{ j } $
		since $ \nu_{ i } = - \nu_{j } $.
		De Giorgi's optimal energy dissipation inequality is also identical.
		
		The compatibility conditions 
		(\ref{varifold_symmetry_of_energy_measures})-(\ref{varifold_symmetry_velocities})
		all hold restricted to $ \Sigma_{ i j } $ $ \hm^{ d - 1 } $-almost 
		everywhere and are therefore satisfied.
		For the compatibility condition 
		(\ref{varifold_mean_curvature_points_in_normal_direction}), we have to 
		prove that the 
		mean curvature vector $ H $ points in the direction of $ \nu_{ i } $ $ 
		\hm^{ d - 1 } $-almost everywhere since $ \langle \lambda_{ t , x , i j } 
		\rangle = \nu_{ i } ( t , 
		x ) $ holds for $ \hm^{ d- 1 } $-almost every $ x $ on $ \Sigma_{ i j } $. 
		But this has been proven by Brakke in 
		\cite[Thm.~5.8]{brakke_kenneth_motion_of_surface_by_mean_curvature}. Note 
		that we apply it for a fixed time t to the varifold $ V = \mu_{ t } $, 
		which is strictly speaking no integer varifold. But all results still hold 
		since 
		\begin{equation*}
		V = 
		\sum_{ 1 \leq i < j \leq P }
		\sigma_{ i j }
		v ( \Sigma_{ i j } )
		\end{equation*}
		is a finite sum. Here $ v ( \Sigma_{ i j } ) $ is the naturally associated 
		varifold to $ \Sigma_{ i j } $ as defined by Brakke.
		The last compatibility condition 
		(\ref{varifold_compatibility_condition_multiphase}) is a consequence of 
		\begin{equation*}
			\abs{ \nabla \chi_{i } }
			=
			\sum_{ i \neq j }
				\hm^{ d - 1 } \llcorner_{ \Sigma_{ i j } },
		\end{equation*}
		which holds since the sets $ \Omega_{ j } $ are a partition of the flat torus.
		Thus we have collected all necessary claims and the proof is finished.
	\end{proof}
	
	\begin{example}
		\label{mean_curvature_vector_does_not_have_to_point_in_normal_direction}
		We want to present an example of a varifold where the mean curvature does 
		not point in normal direction.
		Let $ \flattorus = [ 0 , \Lambda )^{ 2 } $ be the flat torus in two 
		dimensions and let $ \rho \colon \mathbb{ R } \to (-\infty, \infty ) $ be 
		some positive smooth $\Lambda$-periodic function which is not constant. 
		Then we consider the varifold given by 
		\begin{equation*} \mu = \rho ( x_{ 1 } ) \hm^{ 1 } 
		|_{ [ 0 , \Lambda ) \times \{ 0 \} } \otimes \delta_{ e_{ 2 } },
		\end{equation*} 
		where $ 
		e_{ 2 } = ( 0 , 1 )^{ \top } $. We 
		compute that for a given test vector field $ \xi $, we have
		\begin{align*}
		\int
		\inner*{ \diff \xi }{ \mathrm{Id} - p \otimes p }
		\dd{ \mu }
		& =
		\int_{ 0 }^{ \Lambda }
		\inner*{ \diff \xi ( x_{ 1 } , 0 ) }{ e_{ 1 } \otimes e_{ 1 } }
		\rho ( x_{ 1 } )
		\dd{ x_{ 1 } }
		\\
		& = 
		\int_{ 0 }^{ \Lambda }
		\partial_{ x_{ 1 } } \xi^{ 1 } \rho ( x_{ 1 } )
		\dd{ x_{ 1 } }
		\\
		& =
		-
		\int_{ 0 }^{ \Lambda }
		\xi^{ 1 }
		\rho' ( x_{ 1 } )
		\dd{ x_{ 1 } }.
		\end{align*}
		If a mean curvature vector $ H $ would exist, then this integral would have 
		to be equal to
		\begin{equation*}
		-\int_{ 0 }^{ \Lambda }
		\inner*{ H ( x_{ 1 } , 0 ) }{ \xi ( x_{ 1 } , 0 )  }
		\rho ( x_{ 1 } )
		\dd{ x_{ 1 } }.
		\end{equation*}
		This implies that $ H $ would be given on $ [ 0 , \Lambda ) \times \{ 0 \} 
		$ $ 
		\hm^{ 1 } $-almost everywhere by 
		\begin{equation*}
		H ( x ) = \frac{ \rho' ( x_{ 1 } ) }{ \rho ( x_{ 1 } ) } e_{ 1 },
		\end{equation*}
		which does not lie in the normal space of the varifold at points where $ 
		\rho ' ( x_{ 1 } ) \neq 0 $.
	\end{example}
	
	

 	\emergencystretch=1em		
	\printbibliography			
	
\end{document}